\documentclass{amsart}

\usepackage{mathtools}
\usepackage{mathabx} 

\usepackage{mathrsfs,bbm,color}   
\usepackage[scr=boondox]{mathalfa}          

\usepackage{graphicx}
\usepackage{amssymb,enumerate}

\usepackage{thmtools}

\usepackage[colorlinks=true,backref=page,bookmarksopen=false]{hyperref} 

\usepackage{xcolor}
\hypersetup{
    colorlinks,
    linkcolor={red!50!black},
    citecolor={blue!50!black},
    urlcolor={blue!80!black}
}

\renewcommand*{\backref}[1]{}

\renewcommand*{\backrefalt}[4]{\ \tiny 
  \ifcase #1 ({\color{red} \bf NOT CITED.})%
  \or    ($\uparrow$#2)%
  \else   ($\uparrow$#2)%
  \fi}

\declaretheorem[numberwithin=section]{theorem} 
\declaretheorem[sibling=theorem]{corollary} 
\declaretheorem[sibling=theorem]{proposition} 
\declaretheorem[sibling=theorem]{lemma}

\declaretheorem[sibling=theorem]{question}
\declaretheorem[sibling=theorem, style=definition]{definition}
\declaretheorem[sibling=theorem, style=remark]{remark}

\newcommand{\threebar}[1]{{\left\vert\kern-0.25ex\left\vert\kern-0.25ex\left\vert #1 
    \right\vert\kern-0.25ex\right\vert\kern-0.25ex\right\vert}}

\DeclareMathOperator{\Int}{int}

\newcommand{\R}{\mathbb{R}}
\newcommand{\C}{\mathbb{C}}
\newcommand{\A}{\mathsf{A}}
\newcommand{\B}{\mathsf{B}}
\newcommand{\RP}{\mathbb{R}\mathrm{P}}

\newcommand{\CP}{\mathbb{C}\mathrm{P}}

\newcommand{\N}{\mathbb{N}}
\newcommand{\Z}{\mathbb{Z}}

\renewcommand{\setminus}{\smallsetminus}
\renewcommand{\emptyset}{\varnothing}

\title[Poincar\'e-Bendixson theorem for Bebutov shifts]{A Poincar\'e-Bendixson theorem for Bebutov shifts and applications to switched systems}
\author{Jairo Bochi}
\address{Department of Mathematics, The Pennsylvania State University}
\email{\href{mailto:bochi@psu.edu}{bochi@psu.edu}}
\author{Ian D.~Morris}
\address{School of Mathematical Sciences, Queen Mary University of London}
\email{\href{i.morris@qmul.ac.uk}{i.morris@qmul.ac.uk}}
\date{January 9, 2026}


\begin{document}

\begin{abstract}
We prove a version of the Poincar\'e-Bendixson theorem for certain classes of curves on the $2$-sphere which are not required to be the trajectories of an underlying flow or semiflow on the sphere itself.  Using this result we extend the Poincar\'e-Bendixson theorem to the context of continuous semiflows on compact subsets of the $2$-sphere and the projective plane, give new sufficient conditions for the existence of periodic trajectories of certain low-dimensional affine control systems, and give a new criterion for the global uniform exponential stability of switched systems of homogeneous ODEs in dimension three. We prove in particular that periodic asymptotic stability implies global uniform exponential stability for real linear switched systems of dimension three and complex linear switched systems of dimension two. In combination with a recent result of the second author, this resolves a question of R.~Shorten, F.~Wirth, O.~Mason, K.~Wulff and C.~King and resolves a natural analogue of the Lagarias-Wang finiteness conjecture in continuous time.
\end{abstract}

\maketitle

\tableofcontents

\section{Introduction and motivation}

If $\Psi_1,\ldots,\Psi_N$ are continuous vector fields on $\R^d$, we define a trajectory of the \emph{switched system} of differential equations defined by $(\Psi_1,\ldots,\Psi_N)$ to be any absolutely continuous function $x \colon [0,\infty) \to \R^d$ which satisfies a non-autonomous Carath\'eodory differential equation of the form
\begin{equation}\label{eq:caraway}\dot{x}(t)=\sum_{i=1}^N u_i(t)\Psi_i(x(t))\end{equation}
for Lebesgue a.e.\ $t \geq 0$, where each $u_i$ is a Lebesgue measurable function $[0,\infty) \to [0,1]$ and where $\sum_{i=1}^N u_i=1$ Lebesgue almost everywhere. Switched systems of differential equations are a topic of classical and continuing interest in the mathematical theory of control: see for example \cite{ChMaSi25,Li03,LiAn09,ShWiMaWuKi07,SuGe11}. 

In this article the functions $\Psi_i$ will usually be homogeneous or linear, and in this context we will be concerned in particular with the stability of the origin. We recall some standard definitions from this context. We will say that the switched system defined by $(\Psi_1,\ldots,\Psi_N)$ is \emph{globally uniformly exponentially stable} if there exist $C,\kappa>0$ such that $\|x(t)\|\leq Ce^{-\kappa t}\|x(0)\|$ for all $t \geq 0$ and for every trajectory~$x$; \emph{Lyapunov stable} if there exists $C>0$ such that $\|x(t)\|\leq C\|x(0)\|$ for all $t \geq 0$ and for every trajectory $x$; and \emph{periodically asymptotically stable} if $\lim_{t\to\infty}x(t)=0$ for all trajectories such that the underlying switching law $(u_1,\ldots,u_N)$ in \eqref{eq:caraway} is constant or periodic (i.e. such that for some $p>0$ we have $u_i(t+p)=u_i(t)$ a.e.\ for every $i=1,\ldots,N$). We will be especially interested in the case where each $\Psi_i$ is given by a linear map $A_i \colon \R^d \to\R^d$.  The core problem which motivates this article is the following question:
\begin{question}[{\cite[\S4.10]{ShWiMaWuKi07}}, see also \cite{Gu03}]\label{qu:siamreview}
Let $A_1,\ldots,A_N \colon \R^d \to \R^d$ be linear. If the switched system defined by $(A_1,\ldots,A_N)$ is periodically asymptotically stable, is it also globally uniformly exponentially stable?
\end{question}
This question has a well-established analogue in discrete time as follows. Let $M_d(\R)$ denote the vector space of all $d\times d$ real matrices. Given a tuple $\A=(A_1,\ldots,A_N) \in M_d(\R)^N$ let us say that a trajectory of the \emph{discrete-time switched system} defined by $\A$ is any sequence of vectors $(x_n)_{n=1}^\infty$ which satisfies $x_{n+1}\equiv A_{i_n}x_n$ for some sequence $(i_n)_{n=1}^\infty \in \{1,\ldots,N\}^\N$. In a similar manner we call this discrete-time switched system globally uniformly exponentially stable if there exist $C,\kappa>0$ such that $\|x_n\|\leq Ce^{-\kappa n}\|x_1\|$ for all $n \geq 1$ and for all trajectories $(x_n)$, and periodically asymptotically stable if $\lim_{n \to \infty} x_n=0$ for every trajectory $(x_n)$ which is generated by a periodic switching law  $(i_n)$. L.~Gurvits asked in \cite{Gu95} whether every periodically asymptotically stable discrete-time linear switched system is globally uniformly exponentially stable; this question admits the following alternative formulation. Given $\A=(A_1,\ldots,A_N) \in M_d(\R)^N$, define the \emph{joint spectral radius} to be the quantity
\begin{equation}\label{eq:JSR}\varrho(\A)\coloneq \lim_{n\to \infty}\max_{i_1,\ldots,i_n} \left\|A_{i_n}\cdots A_{i_1}\right\|^{\frac{1}{n}}.\end{equation}
One may show (see for example \cite{BeWa92,Ju09}) that the joint spectral radius is equal to the supremum
\begin{equation}\label{eq:gusup}\sup_{(i_n)\text{ periodic}} \lim_{n \to \infty} \left\|A_{i_n}\cdots A_{i_1}\right\|^{\frac{1}{n}}.\end{equation}
If this supremum fails to be a maximum for some $\A$ such that $\varrho(\A)>0$, then the same holds for $\hat{\A}\coloneq \varrho(\A)^{-1}\A$ which additionally satisfies $\varrho(\hat{\A})=1$. This system cannot be globally uniformly exponentially stable since then we would have $\varrho(\hat{\A})<1$, but it is periodically asymptotically stable due to the non-attainment of the supremum. Thus Gurvits' question has a positive answer if and only if the supremum \eqref{eq:gusup} is always attained. That this supremum is always attained was precisely the \emph{finiteness conjecture} made independently by J.C.~Lagarias and Y.~Wang in \cite{LaWa95}. The finiteness conjecture has a trivial positive answer in dimension one, but was shown by Bousch and Mairesse in 2002 to have a negative answer in all higher dimensions: see \cite{BoMa02} and subsequently \cite{BlThVl03,HaMoSiTh11}. 

Similarly to the finiteness conjecture, the answer to Question~\ref{qu:siamreview} varies with the dimension $d$. If $d=1$ then each $1\times1$ matrix $A_i$ may be identified with a real number $a_i$. Under periodic asymptotic stability it is trivial that trajectories with constant switching law $(u_1,\ldots,u_N)$ must converge to zero, so every $a_i$ must be negative, and global uniform exponential stability follows easily with $C=1$ and with $\kappa = \min_i |a_i|$. In higher-dimensional cases one proceeds as follows. It transpires that one may reduce the general $d$-dimensional problem to the special case  in which the linear maps $A_1,\ldots,A_N$ do not have a common invariant subspace other than $\{0\}$ and $\R^d$, and in which $(A_1,\ldots,A_N)$ lies at the boundary of stability in the sense that $(A_1 + \alpha I, \ldots, A_N+\alpha I)$ is globally uniformly exponentially stable for every negative real number $\alpha$ but not for any positive $\alpha$. To resolve this core case of Question~\ref{qu:siamreview} affirmatively it is enough to show that a system with this combination of properties is neither globally asymptotically stable nor periodically asymptotically stable. For such systems it is possible to show that there exists a norm $\threebar{\cdot}$ on $\R^d$, conventionally called a \emph{Barabanov norm}, such that every point in $\R^d$ is the initial point of at least one solution to \eqref{eq:caraway} along which $\threebar{x(t)}$ is constant (see \cite{Ba88b,ChMaSi25}). This immediately implies that the system is Lyapunov stable and not globally uniformly exponential stable, and to complete the argument one must construct a solution to \eqref{eq:caraway} which is generated by a periodic switching law $(u_1,\ldots,u_N)$ and which does not decay to the origin. It is clearly sufficient for the existence of such a solution that there exists a non-injective nonzero trajectory $x$, since for example if $x(0)=x(T)$ then  one may simply periodically repeat the values taken by the underlying switching law $(u_1,\ldots,u_N)$ on the interval $[0,T]$ to obtain a new trajectory $y(t)$ which coincides with $x(t)$ on $[0,T]$ and which is periodic with period $T$, and in particular does not tend to zero.

The core problem to be solved, then, is as follows: given that every point on the unit sphere of some unknown norm $\threebar{\cdot}$ is the initial condition of at least one trajectory of \eqref{eq:caraway} which remains on that unit sphere for all time, prove the existence of a non-injective trajectory which takes values only in that unit sphere. In the case $d=2$ this is rather straightforward: the unit sphere of a norm on $\R^2$ is a topological circle, and an injective trajectory on a topological circle must converge to a point, which is easily shown to be fixed by some convex combination of the linear maps $A_1,\ldots,A_N$. Thus either there exists a non-injective trajectory which performs a periodic circuit around the unit circle of a Barabanov norm, or there exists a constant trajectory; but in either case periodic asymptotic stability fails as required. This leads to a positive answer to Question~\ref{qu:siamreview} in the case $d=2$, and indeed this result has been rediscovered, reanalysed and extended on numerous occasions (see \cite{AgMo23,Bo02,MaLa03,PrMu25,PyRa96,YaXiLe12}). A corollary of this analysis --- which by contrast has been noticed only very recently  --- is that the answer to Question~\ref{qu:siamreview} is \emph{negative} in dimension $4$. To see this one begins by considering a pair of $2\times 2$ matrices $(A_1,A_2)$ which is at the boundary of stability and has a periodic trajectory but no stationary trajectories. Such a system may be shown to have the additional property that every nonzero vector is the initial value of a unique periodic trajectory, and that this trajectory is moreover generated by a unique switching law which is periodic with period $p>0$ independent of the point of origin. The trajectories of the four-dimensional switched system generated by the four matrices of the form $A_i \otimes I + \sqrt{2}\cdot (I\otimes A_j)$ are precisely given by linear combinations of tensor products $x(t)\otimes y(t\sqrt{2})$ where $x(t)$ and $y(t)$ are trajectories of the switched system $(A_1,A_2)$. If the underlying switching law is periodic with period not an integer multiple of $p$ then $x(t)$ converges to zero while $y(t\sqrt{2})$ remains bounded, but if the period \emph{is} an integer multiple of $p$ then $x(t)$ remains bounded while $y(t\sqrt{2})$ converges to zero. Thus all trajectories of this four-dimensional system which are generated by periodic switching laws tend to zero, and the system is periodically asymptotically stable. On the other hand it is easily seen that this tensor product system is not globally uniformly exponentially stable by considering a tensor product $x(t) \otimes x(t\sqrt{2})$ where $x(t)$ is periodic. A detailed version of this argument is presented in \cite{Mo25}. This reasoning immediately implies that the answer to Question~\ref{qu:siamreview} is negative in all dimensions $d \geq 4$.

This leaves the case of Question~\ref{qu:siamreview} in which the dimension is precisely three. By the classical Poincar\'e-Bendixson theorem, any trajectory of the flow defined by a uniquely integrable vector field on the $2$-sphere has either a stationary point or a periodic orbit in its $\omega$-limit set. Based on this intuition we anticipate that every trajectory of \eqref{eq:caraway} which takes values in the unit sphere of a Barabanov norm on $\R^3$ should likewise accumulate on a stationary or self-intersecting trajectory. However, this result is rather beyond the scope of the existing literature. The case of flows defined by continuous vector fields on $S^2$ which are not uniquely integrable is treated in Hartmann's classic textbook \cite{Ha64} (showing in particular that the non-uniqueness of trajectories starting at a given point, as occurs in our context, is not an inherent obstacle in proving this type of theorem) and Poincar\'e-Bendixson-type results for non-autonomous differential inclusions are presented in works such as \cite{Fi88,Fi98}. Since in general the unit sphere of a Barabanov norm is only a Lipschitzian manifold and not a differentiable one, and since the set of directions tangent to this sphere is in general not convex, the hypotheses of the results of \cite{Fi88,Fi98} are not met. On the other hand a plausible line of reasoning is lightly sketched by  E.~Pyatnitskiy and L.~Rapoport in \cite[\S{X}]{PyRa96}, drawing on their earlier result \cite[Theorem~3]{PyRa91}, as follows. Consider a recurrent trajectory $x(t)$ with values in the unit sphere of a Barabanov norm, and assume that the range of $x(t)$ lies in the $\omega$-limit set of some previously-specified trajectory. If $x(t)$ is non-injective then the construction of a periodic orbit originating somewhere along the trajectory $x(t)$ is trivial. Otherwise, the set of possible initial directions of solutions to \eqref{eq:caraway} with initial state $x(0)$ is precisely the convex hull of $\{A_1x(0),\ldots,A_Nx(0)\}$. If this set contains the origin then there exists a stationary trajectory at $x(0)$ and we are done; otherwise, by compactness and convexity this set is necessarily confined to the open half-space defined by some plane $V$ which passes through the origin. This implies that every trajectory of \eqref{eq:caraway} which begins close to $x(0)$ must move transversely to the plane $x(0)+V$, and all such trajectories must enter the same one of the two open half-spaces defined by that plane. The intersection of $x(0)+V$ with the unit sphere of the Barabanov norm thus has the properties required of a \emph{transverse section} in classical proofs of the Poincar\'e-Bendixson theorem and the argument may proceed along more-or-less standard lines, with the same mild modifications as are applied in \cite{Ha64} to accommodate the fact that in general multiple distinct trajectories might emanate from a single point. 

In this work we will take a more general and more fundamental approach to the three-dimensional case of Question~\ref{qu:siamreview} by working at the level of continuous curves on the topological space $S^2$ which are equipped with suitable notions of recurrence and switching. As well as giving a positive answer to Question~\ref{qu:siamreview} in three dimensions --- and thus completing the resolution of a natural continuous-time analogue of the Lagarias-Wang finiteness conjecture ---  this will allow us to consider periodic asymptotic stability of more general homogeneous switched systems and to establish new Poincar\'e-Bendixson-type results for semiflows which have no differentiable regularity and which are defined on subsets of the sphere or projective plane which need not be topological manifolds.  In \S\ref{se:core} below, we state and prove a Poincar\'e-Bendixson theorem for a class of Bebutov shifts which take values on $S^2$ and which admit a notion of switching between different trajectories. In \S\ref{se:app} we present various applications of this result which include, but are not limited to, the positive resolution of Question~\ref{qu:siamreview} in dimension three. This article concludes with some brief remarks on possible extensions of our arguments.

\section{A Poincar\'e-Bendixson Theorem for switched Bebutov shifts}\label{se:core}

\subsection{Principal definitions and results} 
We will consider trajectories of  switched differential equations as curves on $S^2$ in a way which is abstracted from differential equations \emph{per se} by representing them in the form of a Bebutov shift, defined as follows:

\begin{definition}\label{de:Bebutov}
Let $Z$ be a compact metrisable topological space. A \emph{$Z$-valued Bebutov shift} is defined to be a set $\mathfrak{X}$ of continuous functions from the half-line $[0, \infty)$ to $Z$ satisfying the following two properties:
\begin{enumerate}[(i)]
\item\label{it:compactness}
\emph{Compactness:} the set $\mathfrak{X}$ is compact with respect to the compact-open topology on $C([0,\infty), Z)$.
\item\label{it:shift-invariance}
\emph{Shift-invariance:} for every $\phi \in \mathfrak{X}$ and $\tau \geq 0$, the function $\sigma^\tau \phi$ defined by $(\sigma^\tau\phi)(t)\coloneq \phi(\tau+t)$ belongs to $\mathfrak{X}$.
\end{enumerate}
\end{definition}

Since $Z$ is metrisable, the compact-open topology on $C([0,\infty), Z)$ coincides with the topology of uniform convergence on compact sets (see e.g.\ \cite[p.~285]{Mu00}).

\begin{remark}
The study of this class of dynamical systems originates in the work of Bebutov and Kakutani \cite{Be40,Ka68} who considered the problem of embedding arbitrary continuous-time dynamical systems in $C(\R, [0,1])$. Bebutov shifts can be defined in much greater generality by allowing the domain $[0,\infty)$ to be replaced with another topological semigroup $S$; indeed, this wider definition includes the standard context of symbolic dynamical systems in the form of $C(\Z, \{1,\ldots,N\})$ or $C(\N, \{1,\ldots,N\})$ in its compact-open topology.  In this paper we will work exclusively with the semigroup $S = [0,\infty)$.
\end{remark}

We will simply say that $\mathfrak{X}$ is a \emph{Bebutov shift} if it is a $Z$-valued Bebutov shift for some compact metrisable topological space $Z$. 
If $\mathfrak{X}$ is a Bebutov shift then the verification of the following facts is straightforward:  $(\sigma^t)_{t \geq 0}$ defines a continuous semiflow on $\mathfrak{X}$; furthermore, fixed any metric on the space $Z$, the elements of $\mathfrak{X}$ form a uniformly equicontinuous family of functions and, in particular, each of them is uniformly continuous.

In general, a Bebutov shift with values in the $2$-sphere need not contain any periodic or constant functions: for example, the set of functions of the form
\begin{equation}
t \mapsto \left(\sin (t+\alpha) \cos (\sqrt{2}(t+\beta)),  \sin (t+\alpha) \sin (\sqrt{2}(t+\beta)), \cos (t+\alpha)\right) \in S^2
\end{equation}
for $\alpha,\beta \in \R$ is clearly a Bebutov shift with no periodic or constant elements. However, the trajectories of this system are obviously not injective, so if these curves arose as trajectories of the equation \eqref{eq:caraway} then there would be no difficulty in constructing a periodic trajectory simply by repeating the course taken by one of these functions on a minimal interval of non-injectivity. In order to allow us to perform such constructions we will therefore need to consider Bebutov shifts equipped with a notion of switching. 

\begin{definition}\label{de:switched}
A \emph{switched Bebutov shift} is defined to be a Bebutov shift $\mathfrak{X}$ which satisfies the following additional axiom:
\begin{enumerate}[(i)]
\setcounter{enumi}{2}
\item\label{it:concatenation}
\emph{Switching:} If $\phi_1, \phi_2 \in \mathfrak{X}$ satisfy $\phi_1(\tau)=\phi_2(\tau)$ for some $\tau>0$, then the function $\psi \in C([0,\infty), Z)$ defined by $\psi(t)=\phi_1(t)$ for all $t \in [0,\tau]$ and $\psi(t)=\phi_2(\tau)$ for all $t \in [\tau,\infty)$ is an element of $\mathfrak{X}$. 
\end{enumerate}
\end{definition}
We observe that if $\mathfrak{X}$ is a $Z$-valued Bebutov shift and $Z'$ is a closed subset of $Z$, then the set $\mathfrak{X}'\coloneq \{\phi \in \mathfrak{X} \colon \phi(t)\in Z'\text{ for all }t \geq 0\}$
is  a $Z'$-valued Bebutov shift (albeit perhaps an empty one) and if $\mathfrak{X}$ is a switched Bebutov shift then so is $\mathfrak{X}'$.
We will verify in \S\ref{se:app} that the trajectories of switched differential equations in compact subregions of $\R^d$ form switched Bebutov shifts, as do the trajectories of flows and semiflows on arbitrary compact metric spaces. 

The results established in this article will be consequences of the following Poincar\'e-Bendixson theorem for switched Bebutov shifts with values in the $2$-sphere or in the projective plane:

\begin{theorem}\label{th:stronger}
Let $Z$ be a compact  topological space which is homeomorphic to a subset of either $S^2$ or $\RP^2$, and let $\mathfrak{X}$ be a $Z$-valued switched Bebutov shift. If $\phi \in \mathfrak{X}$ is recurrent with respect to the shift semiflow on $\mathfrak{X}$, then there exists $\chi \in \mathfrak{X}$ which is preperiodic and satisfies $\chi(0)=\phi(0)$.  If additionally $\phi$ is injective then we may take $\chi$ constant. \end{theorem}
Clearly Theorem~\ref{th:stronger} implies that such a switched Bebutov shift $\mathfrak{X}$ contains a periodic or constant element, since by the Birkhoff recurrence theorem a continuous semiflow on a compact Hausdorff space must have a minimal set and hence a recurrent point. We note the following stronger articulation of this remark:
\begin{corollary}
Let $Z$ be a compact  topological space which is homeomorphic to a subset of either $S^2$ or $\RP^2$, let $\mathfrak{X}$ be a switched Bebutov shift with values in $Z$, and let $\phi \in \mathfrak{X}$. Then there exists a periodic function $\chi \in \mathfrak{X}$ whose values are confined to the $\omega$-limit set of $\phi$,
\begin{equation}\omega(\phi)\coloneq \bigcap_{T \geq 0}\overline{\phi([T,\infty))}.\end{equation} 
\end{corollary}
\begin{proof}
The set $\mathfrak{X}'\coloneq \{\psi \in \mathfrak{X} \colon \psi(t) \in \omega(\phi)\text{ for all }t\geq 0\}$
is a switched Bebutov shift, and is nonempty since it contains all accumulation points of $\sigma^t \phi$ in the limit as $t \to \infty$. In particular by the Birkhoff recurrence theorem it contains a recurrent element and by Theorem~\ref{th:stronger} it contains a periodic element.
\end{proof}

\begin{remark}
Theorem~\ref{th:stronger} somewhat resembles certain results of V.V.~Filippov which are presented in \cite{Fi98}. Specifically, in \cite[\S16.4]{Fi98} Filippov obtains a version of the Poincar\'e-Bendixson theorem in the context of \emph{autonomous spaces}, a class of objects which roughly correspond to switched Bebutov systems $\mathfrak{X}$ whose elements have domain $\R$ and with the additional property that $\{\phi(t) \colon \phi \in \mathfrak{X}, \text{ }t \in \R\}$ is an open subset of $\R^d$. (The latter hypothesis allows homotopies to be applied to the analysis of trajectories, which is a key method in Filippov's proofs.) Filippov's results are not themselves applicable to any of the problems considered in this article since none of our applications in general satisfy the two additional hypotheses simultaneously. For example, in the context of Barabanov norms of switched linear differential equations on $\R^3$ there in general exist trajectories in the unit sphere of the Barabanov norm which cannot be extended from $[0,\infty)$ to $\R$: see \cite[Proposition 5]{ChGaMa15}. 
\end{remark}

\begin{remark}\label{re:Plykin}
The following more delicate example shows the switching property cannot be eliminated from the hypothesis of Theorem~\ref{th:stronger}, even if it is additionally assumed that all elements of $\mathfrak{X}$ are injective. There exists a smooth diffeomorphism $f \colon S^2 \to S^2$ with four hyperbolic repelling fixed points $p_1,\dots,p_4$ such that the set
$\Lambda \coloneq  \bigcap_{n \ge 0} f^n(S^2 \setminus \{p_1,\dots,p_4\})$ is a hyperbolic attractor with one-dimensional stable bundle, and therefore is laminated by unstable manifolds which are smoothly immersed copies of $\mathbb{R}$. This set $\Lambda$ is called a \emph{Plykin attractor}. Details of the construction can be found in \cite[\S17.2]{KaHa95} and \cite{Co06}; the second reference also contains illustrations. Define $\mathfrak{X}$ to be the set of all smooth unit-speed parametrised curves $\phi\colon [0,\infty) \to S^2$  whose images are contained in $\Lambda$: then $\mathfrak{X}$ is a Bebutov shift and also consists entirely of injective curves, but has no periodic trajectories; the shift semiflow on $\mathfrak{X}$ is even minimal, i.e.\ every orbit is dense. Every path component of $\Lambda$ is traversed in both directions by an element of $\mathfrak{X}$, but trajectories with opposite directions of travel may not be concatenated since the definition of $\mathfrak{X}$ does not permit trajectories to instantaneously reverse their direction of travel. Thus, $\mathfrak{X}$ does not satisfy the switching property of Definition~\ref{de:switched} and is not a switched Bebutov shift. The open set $S^2\setminus \Lambda$ has three connected components with a common topological boundary, thus forming \emph{Lakes of Wada}: see \cite{Co06}. 
\end{remark}

To prove Theorem~\ref{th:stronger} we adopt the following strategy. If $\phi$ is non-injective, then by a careful but intuitively straightforward application of the axioms of Definitions \ref{de:Bebutov} and \ref{de:switched} we may construct the desired function $\chi$ by periodically repeating the trajectory of $\phi$ along a minimal interval of non-injectivity. If the switched Bebutov shift $\mathfrak{X}$ admits trajectories which can spend arbitrarily long times in arbitrarily small neighbourhoods of $\phi(0)$ then an easy compactness argument shows that the function with constant value $\phi(0)$ belongs to $\mathfrak{X}$, and we are done.  The third and most difficult case of Theorem~\ref{th:stronger} is that in which $\phi$ is injective and there exist $\delta,T>0$ such that $\phi(t)$ cannot remain entirely in the $\delta$-neighbourhood of $\phi(0)$ for any interval of duration at least $T$. In this case we will show that there exist intervals $[t_1,t_2]$ and $[t_3,t_4]$ such that $\phi$ traverses an arc in $B_\delta(\phi(0))$ during the interval $[t_1,t_2]$, and then traces another similarly small arc in $B_\delta(\phi(0))$ during the interval $[t_3,t_4]$, with the endpoints of the second arc being arbitrarily close to those of the first arc but in reversed order. The defining property of $\delta$ implies that neither interval can be longer than $T$, so by taking limits we may find $\phi_1, \phi_2 \in \mathfrak{X}$ and points $x_1,x_2 \in Z$ such that $\phi_1$ traces an arc from $x_1$ to $x_2$ and $\phi_2$ traces an arc from $x_2$ to $x_1$, each along an interval of length not greater than $T$. By appeal to Definition~\ref{de:switched} we may deduce the existence of a trajectory in $\mathfrak{X}$ which oscillates periodically between $x_1$ and $x_2$ and is confined to the open $\delta$-ball around $\phi(0)$; but since $\delta>0$ can freely be taken arbitrarily small, it follows by taking the limit $\delta \to 0$ that the constant trajectory $\phi(0)$ is an element of $\mathfrak{X}$ after all, contradicting the supposition that we are in the third case. The construction underlying this part of this argument is captured in the following statement (illustrated in Figure~\ref{fig:reversed-arc}):
\begin{theorem}\label{th:core}
Let $\phi \colon [0,\infty) \to S^2$ be continuous and injective, and suppose that $(\sigma^t\phi)_{t \geq 0}$ accumulates on $\phi$ with respect to the compact-open topology on $C([0,\infty), S^2)$  as $t \to \infty$. Then for every $\delta>0$ the following holds: there exist distinct points $x_1, x_2 \in S^2$ such that for all $\varepsilon>0$ we may choose intervals $[t_1,t_2]$ and $[t_3,t_4]$ such that
\begin{equation}
d(\phi(t_1),x_1) < \varepsilon, \quad 
d(\phi(t_2),x_2) < \varepsilon, \quad 
d(\phi(t_3),x_2) < \varepsilon, \quad 
d(\phi(t_4),x_1) < \varepsilon,
\end{equation}
and
\begin{equation}\label{eq:delta-bound}\max_{t \in [t_1, t_2] \cup [t_3, t_4]} d(\phi(t),x_1) < \delta.\end{equation}
Furthermore we may take $x_1\coloneq \phi(T)$ for any desired value of $T\geq 0$.
\end{theorem}

\begin{figure}
	\includegraphics{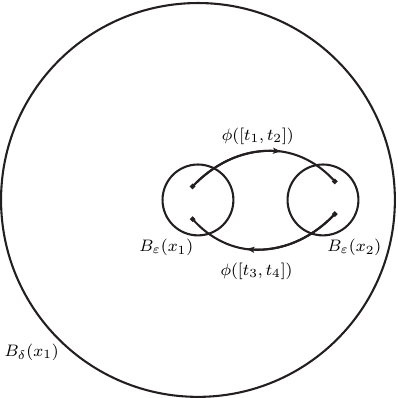}
	\caption{The two pieces of trajectory produced by Theorem~\ref{th:core}.}\label{fig:reversed-arc}
\end{figure}

\begin{remark}
It will be seen in the proof that once $\delta>0$ has been chosen we may also take $x_2\coloneq \phi(T')$ where $T'>T$ is arbitrarily close to $T$. It follows easily from recurrence that the interval $[t_3,t_4]$ may be chosen disjoint from $[t_1,t_2]$ if desired. 
\end{remark}

Theorem~\ref{th:core} is of course false if $S^2$ is replaced with more general manifolds; for example if $\phi$ is instead allowed to be a straight trajectory $\phi(t)\coloneq (t, \alpha t)$ on the torus $\R^2/\Z^2$ with irrational slope $\alpha$. 

Let $S^1$ denote the unit circle, endowed with an orientation. 
If $a$ and $b$ are two distinct points in $S^1$, the \emph{interval} $[a,b]$ denotes the image of any positively oriented simple curve on $S^1$ from $a$ to $b$.
The proof of Theorem~\ref{th:core} will ultimately rely on the following topological fact:

\begin{proposition}\label{pr:three-discs-lemma}
Let $D_1, D_2, D_3 \subset S^2$ be pairwise disjoint topological closed discs and $\psi \colon S^1 \to S^2$ be a Jordan curve. If there exist two disjoint intervals $[a_1,b_1], [a_2,b_2] \subset S^1$ satisfying $\psi(a_i) \in \Int D_1$, $\psi(b_i) \in \Int D_2$ and $\psi([a_i,b_i]) \cap D_3=\emptyset$, then there exists an interval $[c,d]\subset S^1$ satisfying $\psi(c) \in \Int D_2$, $\psi(d) \in \Int D_1$ and $\psi([c,d]) \cap D_3 = \emptyset$.
\end{proposition}

\begin{remark}\label{rem:from_R2_to_S2}
To prove Proposition~\ref{pr:three-discs-lemma} it is sufficient to prove the same statement with $\R^2$ replaced throughout by $S^2$, since the proper statement of the proposition may be deduced from this planar version by puncturing the sphere in a single point not belonging to $D_1 \cup D_2 \cup D_3 \cup \psi(S^1)$. Our proof of Proposition \ref{pr:three-discs-lemma} will therefore assume that the discs $D_1, D_2, D_3$ and curve $\psi(S^1)$ belong to the plane and not the sphere. An instance of this version of the proposition is illustrated in Figure~\ref{fig:three-discs}.
\end{remark}

\begin{figure}
	\includegraphics{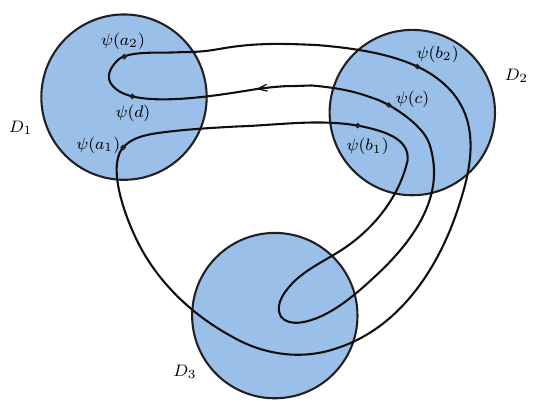}
	\caption{An example in the setting of Proposition~\ref{pr:three-discs-lemma}. The existence of two arcs which begin in $D_1$, end in $D_2$ and do not pass through $D_3$ guarantees the existence of an arc making the same journey in the reverse direction which also does not pass through $D_3$.}\label{fig:three-discs}
\end{figure}

The remainder of this section is organised as follows: 
A proof of Proposition~\ref{pr:three-discs-lemma} is given in \S\ref{ss:proof_three_disks}.
From there, we derive our main statements as follows:
\begin{center}
Proposition~\ref{pr:three-discs-lemma}
\quad $\xRightarrow{\text{\S\ref{ss:proof_core}}}$ \quad
Theorem~\ref{th:core} 
\quad $\xRightarrow{\text{\S\ref{ss:proof_stronger}}}$ \quad 
Theorem~\ref{th:stronger} 
\end{center}

\subsection{A foray into plane topology} \label{ss:proof_three_disks}

\subsubsection{Overview}\label{sss.overview}

Suppose we are given a finite family of smooth closed curves on a smooth surface which are in general position, that is, are pairwise transverse and present no triple intersections. By travelling around each of these curves (according to some preferred orientation) and keeping track of the other curves we meet along the way, we obtain some combinatorial data, namely a family of cyclic words. The topology of the surface somehow constrains these data. For instance, if we have four smooth curves on the sphere, three of which bound disjoint disks, then Proposition~\ref{pr:three-discs-lemma} provides a non-trivial restriction on the manner in which the fourth curve intersects the three others. 
We were not able to find in the literature general information about what intersection patterns are possible (though problems of similar flavour were considered by some authors, e.g.\ \cite{Pe76,JMM96}). 

Therefore we provide a direct proof of Proposition~\ref{pr:three-discs-lemma}, in the planar version noted in Remark \ref{rem:from_R2_to_S2}.
Though our argument is intuitively simple, in order to make it precise it is convenient to use the language of graph theory. It is also convenient to work with ``polygonal'' plane graphs, which we now define.

\subsubsection{Preliminaries on plane graphs}

A \emph{plane graph} is a finite graph $G$ embedded in $\R^2$ in such a way that every edge is represented by a piecewise affine simple curve, called an \emph{arc}. Every arc connects two distinct vertices (its endpoints) and any two arcs are either disjoint or intersect at only one vertex. The \emph{faces} of a plane graph $G$ are the connected components of the open set $\R^2 \setminus G$. Every plane graph has a unique unbounded face. A plane graph $G'$ is a \emph{subgraph} of a plane graph $G$ if every vertex of $G'$ is a vertex of $G$ and every arc of $G'$ is an arc of $G$. In this case, every face of $G$ is contained in a unique face of $G'$.

The following proposition collects three useful properties which correspond to Lemma~4.2.2, Proposition~4.2.6, and Theorem~4.2.9 in \cite{Di25}.
Recall that a \emph{cycle} is a graph where every vertex has degree $2$.

\begin{proposition}\label{pr:faces}
Let $G \subset \R^2$ be a connected plane graph whose vertices all  have degree at least $2$.
Then:
\begin{enumerate}[(i)]
\item\label{it:boundary_cycle}  
The boundary of each face of $G$ is a cycle subgraph.
\item\label{it:two_faces} 
Every arc of $G$ is part of the boundary of exactly two faces of $G$.
\item\label{it:Euler} 
If $G$ has $n$ vertices and $m$ arcs, then it has $m-n+2$ faces (Euler's formula).
\end{enumerate}
\end{proposition}

An \emph{orientation} on an arc $I \subset \R^2$ consists of specifying one endpoint as the \emph{tail} and the other as the \emph{head}. A \emph{positive parametrisation} of an oriented arc $I$ is an injective piecewise affine map $\gamma$ from an interval $[a,b] \subset \R$ onto $I$ such that $\gamma(a)$ is the tail of $I$ and $\gamma(b)$ is the head of $I$.
Orientations of arcs can be used to distinguish between the two faces provided by part~\eqref{it:two_faces} of Proposition~\ref{pr:faces}, as follows:

\begin{proposition}\label{pr:left_face}
Let $G$ be a connected plane graph whose vertices all have degree at least $2$.
Let $I$ be an arc of $G$, and endow it with an orientation.
Then there exists a unique face $F$ of $G$ whose closure contains $I$ and which satisfies the following property:
for every piecewise affine parametrisation $\gamma \colon [a,b] \to I$ such that $\gamma(a)$ is the tail of $I$, every parameter $t \in (a,b)$ for which the velocity $\gamma'(t)$ exists, and every point $q$ in $F$ sufficiently close to $\gamma(t)$, the ordered pair of vectors
$\big( \gamma'(t) , q- \gamma(t) \big)$ is positively oriented\footnote{We say that an ordered pair $(v_1,v_2)$ of vectors $v_i=(x_i,y_i) \in \R^2$ is \emph{positively oriented} if $x_1 y_2 - x_2 y_1 > 0$.}.
\end{proposition}
\begin{proof}
Given such a parameterisation $\gamma$ it follows directly from Proposition \ref{pr:faces}\eqref{it:two_faces} that there exists a face $F$ which makes the pair $\big( \gamma'(t) , q- \gamma(t) \big)$ positively oriented for all $q \in F$ close enough to $\gamma(t)$, for at least one value of $t$. That this face has the same property for all times of differentiability $t \in [a,b]$ and for all parameterisations follows by an easy connectedness argument.
\end{proof}

In the situation of Proposition~\ref{pr:left_face}, we say that $F$ as above is the face of $G$ to the \emph{left} of the oriented arc $I$.

An \emph{orientation} for a cycle plane graph $C$ consists of a choice of orientation for each arc in such a way that each vertex is the tail of one arc and the head of another. Clearly, $C$ admits exactly two orientations: orienting one of its arcs determines the orientation on the others.

\begin{proposition}\label{pr:boundary_orientation}
Let $G$ be a connected plane graph whose vertices all have degree at least $2$.
Let $F$ be a face of $G$, and let $C = \partial F$ be its boundary cycle.
Then $C$ admits a unique orientation such that $F$ sits to the left of every arc of $C$.
\end{proposition}

\begin{proof}
Choose arbitrarily an arc $I$ in $C$. It follows easily from Proposition~\ref{pr:left_face} and Proposition~\ref{pr:faces}\eqref{it:two_faces} that exactly one of the two orientations on $I$ places $F$ to the left of $I$, and this orientation clearly has a unique extension to the whole of $C$. It remains only to show that $F$ is to the left of every arc of $C$ with respect to this orientation. It is easily verified using connectedness that if $J_1$ and $J_2$ are consecutive arcs of $C$ with respect to this orientation, and $F$ lies to the left of $J_1$, then $F$ also lies to the left of $J_2$. The result follows by induction on the arcs of~$C$. 
\end{proof}

In the situation of Proposition~\ref{pr:boundary_orientation}, we say that the face $F$ sits to the \emph{left} of the oriented cycle~$C$.
If $G$ itself is a cycle then it has only two faces, one of which sits to the left of the cycle.

The following property is an immediate consequence of our definitions:

\begin{proposition}\label{pr:two_left_faces}
Let $G$ be a connected plane graph whose vertices all have degree at least $2$.
Let $G'$ be a subgraph of $G$ with the same properties.
Let $I$ be an oriented arc of $G'$ (and therefore also of $G$).
Let $F$ (respectively $F'$) be the face of $G$ (respectively $G'$) to the left of $I$.
Then $F \subseteq F'$.
\end{proposition}

Finally, we state the following lemma, whose proof is immediate:

\begin{lemma}\label{le:captain_obvious}
Let $C$ be an oriented cycle whose vertices are labelled either $1$ and~$2$, with each label occurring at least once. Then $C$ must contain an arc with tail labelled~$1$ and head labelled $2$, and another arc with tail labelled $2$ and head labelled~$1$.
\end{lemma}

\subsubsection{A key lemma}

The following result forms the skeleton of the proof of Proposition~\ref{pr:three-discs-lemma}:
\begin{lemma}\label{le:two12one21}
Let $G$ be a plane graph which is the union of four cycle subgraphs $\Gamma_1$, $\Gamma_2$, $\Gamma_3$ and $\Theta$.
Suppose that the bounded faces of the $\Gamma_i$'s have disjoint closures, and that every vertex of $\Theta$ is also a vertex of one of the $\Gamma_i$'s.
Fix an orientation on the cycle $\Theta$.
If $\Theta$ includes two distinct arcs with tail in $\Gamma_1$ and head in $\Gamma_2$, then it also includes an arc with tail in $\Gamma_2$ and head in $\Gamma_1$. 
\end{lemma}

\begin{proof}
Let $D_i$ denote the bounded face of the cycle subgraph $\Gamma_i$, for $i=1,2,3$.
Let $E$ be the (perhaps unbounded) face of the oriented cycle $\Theta$ sitting to its left. 
Let $I$ and $J$ be two distinct arcs of $\Theta$ with tails in $\Gamma_1$ and heads in $\Gamma_2$. 

Consider the plane graph $G' \coloneq \Gamma_1 \cup \Gamma_2 \cup I \cup J$.
Then $G'$ is connected and has two more arcs than vertices; furthermore, every vertex has degree either $2$ or~$3$.
Thus, by Euler's formula (part~\eqref{it:Euler} of  Proposition~\ref{pr:faces}), $G'$ has four faces. 
Two of them are $D_1$ and $D_2$, and these are the faces that do not contain $I$ nor $J$ as part of their boundaries. 
There is a third face which contains $\overline{D_3}$, and a fourth face which is disjoint from $\overline{D_3}$. Let $F'$ be that fourth face. 

Consider the boundary $\partial F'$, which is a cycle in $G'$ by Proposition~\ref{pr:faces}\eqref{it:boundary_cycle}.
We orient this cycle in such a way that the face $F'$ sits to its left (recall Proposition~\ref{pr:boundary_orientation}).
There are exactly two arcs of $\partial F'$ which are not contained in $\Gamma_1 \cup \Gamma_2$, namely $I$ and $J$.
An application of Lemma~\ref{le:captain_obvious} shows that exactly one of these two arcs has tail in $\Gamma_1$ with respect to the orientation on $\partial F'$. We can assume that $I$ is the arc with this property, by renaming if necessary. 
This means that the orientations of cycles $\Theta$ and $\partial F'$ agree on the common arc $I$ and  disagree on the other common arc $J$.
Therefore, the face $F'$ (of the graph $G'$) and the face $E$ (of the graph $\Theta$) both sit to the left of $I$.

Since $G$ is a union of cycles, every vertex has degree at least $2$.
Let $F$ be the face of $G$ sitting to the left of the arc $I$, with respect to the orientation which $I$ shares with both $\Theta$ and $\partial F$. 
Since both $\Theta$ and $G'$ are subgraphs of $G$, Proposition~\ref{pr:two_left_faces} gives:
\begin{equation}\label{e.face_inclusion}
F \subseteq E \cap F' \, .
\end{equation}
We orient the cycle $\partial F$ so that $F$ sits to its left.
This cycle is contained in the closure of $F'$, which is disjoint from $\overline{D_3}$. 
Recalling the assumption that every vertex of $G$ is a vertex of one of the $\Gamma_i$'s, we see that every vertex of $\partial F$ is a vertex of either $\Gamma_1$ or $\Gamma_2$. 
Furthermore, the oriented cycle $\partial F$ contains the oriented arc $I$, whose tail is in $\Gamma_1$ and whose head is in $\Gamma_2$. Lemma~\ref{le:captain_obvious} implies that $\partial F$  must contain an oriented arc $K$ with tail in $\Gamma_2$ and head in $\Gamma_1$. 

Let $\tilde{E}$ be the face of $\Theta$ sitting to the left of the oriented arc $K$.
Since $F$ is the face of $G$ sitting to the left of the same arc,
Proposition~\ref{pr:two_left_faces} yields: 
\begin{equation}
F \subseteq \tilde{E} \, .
\end{equation}
So \eqref{e.face_inclusion} implies that $E \cap \tilde{E} \neq \emptyset$.
Since $E$ and $\tilde E$ are faces of the same graph $\Theta$, they must coincide: $\tilde{E} = E$.
This proves that $E$ sits to the left of $K$.
That is, the orientation on the arc $K$ is compatible with the orientation on $\Theta$.
We have found an arc of $\Theta$ with tail in $\Gamma_2$ and head in $\Gamma_1$, and the proof is concluded.
\end{proof}

\subsubsection{Polygonal approximation}

It is convenient to use $\R/\Z$ as a model of the unit circle.
A continuous closed curve $\psi \colon S^1 \to \R^2$ is called \emph{polygonal} if it is piecewise affine. 

\begin{lemma}\label{le:polygonal}
If $\psi \colon S^1 \to \R^2$ is a Jordan curve, then for every $\delta>0$ there exists a polygonal Jordan curve $\hat \psi \colon S^1 \to \R^2$ such that $\sup_{t \in S^1} \|\psi(t)-\hat\psi(t)\|<\delta$. 
\end{lemma}

This lemma is due to Jordan~\cite[\S 40, pp.~588--589]{Jordan}. Other proofs can be found in \cite[Lemma~2]{Tv80} and \cite[Lemma~9]{Hales07}.

\subsubsection{Derivation of Proposition~\ref{pr:three-discs-lemma}}
 
We now use Lemmas \ref{le:two12one21} and \ref{le:polygonal} to prove the planar analogue of Proposition~\ref{pr:three-discs-lemma} described in Remark~\ref{rem:from_R2_to_S2}, from which the proper statement of the proposition follows directly.

Let $\Gamma_1$ and $\Gamma_2$ be polygonal closed curves in $\R^2$ which enclose bounded faces $F_1$ and $F_2$ such that $\psi(a_1),\psi(a_2) \in F_1$,  $\psi(b_1), \psi(b_2) \in F_2$, and $\overline{F_i}\subseteq \Int D_i$ for $i=1,2$. Let $\Gamma_3$ be a polygonal closed curve in $\R^2$ enclosing a bounded face $F_3\supset D_3$ whose closure does not intersect $\psi([a_1,b_1]) \cup \psi([a_2,b_2]) \cup D_1 \cup D_2$. Choose $\delta>0$ such that the closed $\delta$-ball around every point of $\Gamma_i$ is contained in $\Int D_i$ if $i=1,2$ or in $\R^2 \setminus D_3$ if $i=3$. 
Using Lemma~\ref{le:polygonal}, we choose a (parametrized) polygonal path $\hat\psi \colon S^1 \to \R^2$ which is uniformly $\delta$-close to $\psi$. Perturbing $\hat{\psi}$ if necessary, we can assume that its image $\Theta \coloneq \hat{\psi}(S^1)$ has finite intersection with each of the polygonal closed curves $\Gamma_i$. Endow the set $G \coloneq \Gamma_1 \cup \Gamma_2 \cup \Gamma_3 \cup \Theta$ with the structure of a plane graph by declaring that each point of $(\Gamma_1 \cup \Gamma_2 \cup \Gamma_3) \cap \Theta$ is a vertex, and by adding extra vertices to the $\Gamma_i$'s (but not to $\Theta$) in order to eliminate any arcs with the same set of endpoints. For each $i=1,2$, the set $\psi([a_i,b_i])$ contains an arc with tail in $\Gamma_1$ and head in $\Gamma_2$. These two arcs are distinct, so we may apply Lemma~\ref{le:two12one21} to deduce the existence of an arc in $\Theta$ with tail in $\Gamma_2$ and head in $\Gamma_1$. This asserts the existence of an interval $[c,d]\subset S^1$ such that $\hat\psi(c) \in \Gamma_2$, $\hat\psi(d) \in \Gamma_1$ and $\hat\psi([c,d]) \cap \overline{F_3}=\emptyset$. By the definition of $\delta$ and $\hat\psi$ this implies $\psi(c) \in \Int D_2$, $\psi(d) \in \Int D_1$ and $\psi([c,d])\cap D_3 = \emptyset$ as required. The proof of Proposition~\ref{pr:three-discs-lemma} is now complete.

\subsection{Proof of Theorem~\ref{th:core}} \label{ss:proof_core}

We will prove Theorem~\ref{th:core} in the special case $T=0$ from which the general statement follows by considering $\sigma^T\phi$ in place of $\phi$. We equip $S^2$ with the angle metric $d$. 
Define $x_1\coloneq \phi(0) \in S^2$ and without loss of generality suppose that $\delta>0$ is small enough that $D_3\coloneq S^2 \setminus B_\delta(x_1)$ is a topological disc. Let $T_3>0$ be a real number which is small enough that $\phi([0,T_3])$ is contained in the open ball $B_\delta(x_1)$, and define $x_2\coloneq \phi(T_3)$. Obviously $x_1$ and $x_2$ are distinct.  Without loss of generality we suppose that $\varepsilon \in (0,\delta/2)$ is small enough that the closed $\varepsilon$-balls around $x_1$ and $x_2$ are disjoint subsets of the open $\delta$-ball centred at $x_1$, and we denote these balls respectively by  $D_1$ and $D_2$. Clearly $D_1$, $D_2$ and $D_3$ are pairwise disjoint topological discs in $S^2$. Let $T_2$ be the latest time $t\in (0,T_3)$ such that $\phi(t) \in \partial D_1$. 
Since $\phi$ is recurrent in the compact-open topology we may choose $\tau>0$ such that
\begin{equation}\sup_{t \in [0,T_3]} d(\phi(t),\phi(t+\tau))<\varepsilon,\qquad \sup_{t \in [0,T_3]} d(\phi(t+\tau),x_1)<\delta.\end{equation}
Define $T_4\coloneq \tau$ and $T_5\coloneq T_3+\tau$, and note that $\phi(T_4) \in \Int D_1$, $\phi(T_5) \in \Int(D_2)$ and $\phi([T_4,T_5]) \cap D_3 =\emptyset$. Since $x_1 \notin \phi([T_2,T_5])$ we may choose $\kappa \in (0,\varepsilon)$ small enough that $\overline{B_\kappa(x_1)}$ is disjoint from $\phi([T_2,T_5])$. Now define $T_1$ to be the latest time  $t\in (0,T_2)$ such that $d(\phi(t),x_1)=\kappa$, and let $T_6$ be the earliest time $t>T_5$ at which $d(\phi(t),x_1)=\kappa$, the existence of which  is guaranteed by recurrence. Define $\psi \colon [0,T_6] \to S^2$ by $\psi(t)=\phi(t)$ for all $t \in [T_1,T_6]$, and by taking $\psi$ on $[0,T_1]$ to be an injective path along $\partial B_\kappa(x_1)$ which begins at $\phi(T_6)$ and ends at $\phi(T_1)$. Since by construction $d(\phi(t),x_1)\neq \kappa$ for all $t \in (T_1,T_6)$, the function $\psi$ thus defined is injective, continuous, and satisfies $\psi(0)=\psi(T_6)$ and  $\psi([T_4,T_5]) \cap D_3 =\emptyset$.

We now identify $[0,T_6]$ with $S^1$ by identifying $0$ with $T_6$ and view $\psi$ as a continuous injection $S^1 \to S^2$ in the obvious manner. It follows from Proposition~\ref{pr:three-discs-lemma} that there exists an interval $[a,b]\subset S^1$ such that $\psi(a) \in \Int D_2$, $\psi(b)\in \Int D_1$ and $\psi([a,b]) \cap D_3=\emptyset$. Since $\psi(t) \in \Int D_1$ for all $t \in [0,T_1]$ we cannot have $a \in [0,T_1]$, and since $\phi(t) \in \Int D_1$ for all $t$ in a small neighbourhood of $T_6$, by moving the right endpoint of $[a,b]$ leftwards if necessary we may assume without loss of generality that $b$ also does not lie in $[0,T_1]$. Thus $[a,b]\subset (T_1,T_6)$ and in particular $\psi(t)=\phi(t)$ for all $t \in [a,b]$. The properties $\psi(a) \in \Int D_2$, $\psi(b)\in \Int D_1$ and $\psi([a,b]) \cap D_3=\emptyset$ directly assert that
\begin{equation}d(\psi(a),x_2)<\varepsilon, \qquad d(\psi(b),x_1)<\varepsilon, \qquad \sup_{t \in [a,b]} d(\psi(t),x_1) <\delta\end{equation} 
and the same inequalities therefore hold for $\phi$ in place of $\psi$. Taking  $t_1\coloneq 0$, $t_2\coloneq T_3$, $t_3\coloneq a$, $t_4\coloneq b$ proves the theorem.

\subsection{Proof of Theorem~\ref{th:stronger}} \label{ss:proof_stronger}

While Theorem~\ref{th:core} is the main ingredient in the proof of Theorem~\ref{th:stronger}, we will also need a couple of extra preparations. The following two simple lemmas are valid for switched Bebutov shifts taking values in any compact metrisable space $Z$, and compensate for the weakness of the axioms specified in Definitions~\ref{de:Bebutov}--\ref{de:switched}. The first lemma provides a form of switching between an infinite sequence of trajectories:

\begin{lemma}\label{le:multiple_switching}
Let $\mathfrak{X}$ be a $Z$-valued switched Bebutov shift.
Let $(\psi_n)_{n=1}^\infty$ be a sequence in $\mathfrak{X}$.
Assume there exists a strictly increasing and unbounded real sequence $(\tau_n)_{n=0}^\infty$ with $\tau_0 = 0$ such that
\begin{equation}\label{eq:matching}
\psi_n(\tau_n) = \psi_{n+1}(\tau_n) \quad \text{for each $n$.}
\end{equation}
Let $\xi \colon [0, \infty) \to Z$ be the function that coincides with $\psi_n$ on the interval $[\tau_{n-1},\tau_{n}]$, for each $n \ge 1$.
Then $\xi \in \mathfrak{X}$.
\end{lemma}

\begin{proof}
An inductive application of the switching axiom in Definition~\ref{de:switched} shows that for each $n\ge 1$, there exists $\xi_n \in \mathfrak{X}$ which coincides with $\xi$ on $[0,\tau_n]$ and coincides with $\psi_{n+1}$ on $[\tau_n,\infty)$. Note that $\xi_n$ tends to $\xi$ in the compact-open topology. The compactness axiom in Definition~\ref{de:Bebutov} now guarantees that $\xi \in \mathfrak{X}$.
\end{proof}

The shift-invariance property in Definition~\ref{de:switched} permits us only to translate functions leftwards and not rightwards. The next lemma remediates this difficulty, under the hypothesis of recurrence.

\begin{lemma}\label{le:inverse_branches}
Let $\mathfrak{X}$ be a $Z$-valued switched Bebutov shift, and let $\phi \in \mathfrak{X}$ be recurrent with respect to the shift semiflow $(\sigma_t)_{t \ge 0}$. Then for any $s \ge 0$ there exists $\psi \in \mathfrak{X}$ such that $\sigma^s \psi = \phi$.
\end{lemma}

\begin{proof}
By assumption, there exists a sequence $(t_n)_{n=1}^\infty$ increasing to infinity such that $\sigma^{t_k} \phi$ converges to $\phi$ in the compact-open topology. By passing to a subsequence, we can assume that $t_1 \ge s$ and that $\sigma^{t_k - s} \phi$ converges to some $\psi \in \mathcal{X}$. Then $\sigma^s \psi = \phi$.
\end{proof}

We now commence the proof of Theorem~\ref{th:stronger}. 
We will first consider the case of a non-injective trajectory $\phi$.
In this case, the topological hypotheses on $Z$ are superfluous.

\begin{proof}[Proof of Theorem~\ref{th:stronger} in the non-injective case]
Suppose $\phi$ is a recurrent and non-injective trajectory of the Bebutov shift $\mathfrak{X}$.  
Let $T \ge 0$ and $L > 0$ be such that $\phi(T)=\phi(T+L)$.
Let $\psi_1 \coloneq \phi$.
For each $n \ge 2$, using Lemma~\ref{le:inverse_branches} we find $\psi_n \in \mathfrak{X}$ such that $\sigma^{(n-1)L} \psi_n = \phi$. Let $\tau_n \coloneq T + nL$ for $n \ge 1$. Then condition~\eqref{eq:matching} in Lemma~\ref{le:multiple_switching} is satisfied. An application of that lemma produces a preperiodic $\chi \in \mathfrak{X}$ such that $\chi(0) = \phi(0)$.
\end{proof}

We will first complete the proof of Theorem~\ref{th:stronger} in the case where $Z$ is homeomorphic to a subset of the sphere, and then apply a lifting argument to derive the case in which $Z$ is homeomorphic to a subset of the projective plane. In these arguments it is clearly sufficient to assume that $Z$ is precisely equal to a compact subset of the sphere or of the projective plane respectively.

\begin{proof}[Proof of Theorem~\ref{th:stronger} for subsets of the sphere]
Let $\mathfrak{X}$ be a switched Bebutov shift with values in a compact set $Z\subseteq S^2$, and fix a recurrent element $\phi \in \mathfrak{X}$ throughout the proof. The case of non-injective $\phi$ was already dealt with, so assume that $\phi$ is injective. Our goal then is to prove that the constant function which takes only the value $\phi(0)$ is an element of $\mathfrak{X}$. We argue by contradiction: suppose that the constant $x_1 \coloneq \phi(0)$ is not in $\mathfrak{X}$. Since $\mathfrak{X}$ is compact with respect to the compact-open topology, it follows that there exists numbers $r>0$ and $\delta>0$ with the following property:
\begin{equation}\label{eq:non_constant}
\text{for all $\psi \in \mathfrak{X}$,} \quad \max_{t \in [0,r]}  d(\psi(t), x_1) > \delta \, .
\end{equation}

We first claim that for every $a \ge 0$, there exists $\psi \in \mathfrak{X}$ and $p \in (0,r)$ such that 
\begin{equation}\label{eq:first_claim}
\psi(a) = \psi(a+p) = x_1 \quad \text{and} \quad \max_{t \in [a,a+p]} d(\psi(t),x_1) \le \delta \, .
\end{equation}
Indeed, let $x_1\coloneq \phi(0)$ and $x_2 \in S^2$ be the distinct points given by Theorem~\ref{th:core} applied to the function $\phi$ and the real number $\delta$. By Theorem~\ref{th:core}, for every $n \geq 1$ there exist subintervals $[t_{1,n}, t_{2,n}]$ and $[t_{3,n}, t_{4,n}]$ of $[0,\infty)$ such that
\begin{alignat}{2}
d(\phi(t_{1,n}),x_1) &< \frac{1}{n}, &\qquad d(\phi(t_{2,n}),x_2) &< \frac{1}{n} \, ,\\
d(\phi(t_{3,n}),x_2) &< \frac{1}{n}, &\qquad d(\phi(t_{4,n}),x_1) &< \frac{1}{n} \, ,
\end{alignat}
and
\begin{equation}
\max_{t \in [t_{1,n}, t_{2,n}] \cup [t_{3,n}, t_{4,n}]} d(\phi(t),x_1) < \delta \, ,
\end{equation}
and using the recurrence of $\phi$ we assume without loss of generality that $t_{1,n} \geq a$ and $t_{3,n} \geq a+r$ for every $n \geq 1$. 
Property~\eqref{eq:non_constant} implies that $t_{4,n}-t_{3,n} <r$ and $t_{2,n}-t_{1,n} < r$ for every $n \geq 1$.
Translating $\phi$ by $t_{1,n}-a$ it follows that for every $n \geq 1$ we may find $\psi'_n \in \mathfrak{X}$ such that 
\begin{equation}
d\left(\psi'_n(a), x_1\right)<\frac{1}{n},\qquad \min_{t \in [a,a+r]} d\left(\psi'_n(t), x_2\right)<\frac{1}{n} \, .
\end{equation}
Fix an accumulation  point $\psi' \in \mathfrak{X}$  of $(\psi'_n)_{n=1}^\infty$ and note that $\psi'(a)=x_1$ and $\psi'(a+t')=x_2$ for some $t' \in [0,r]$, and additionally $\psi'(t) \in \overline{B_\delta(x_1)}$ for all $t \in [a,a+t']$. Obviously $t' \neq 0$ since the points $x_1$ and $x_2$ are distinct. Translating $\phi$ by $t_{3,n}-a-t'$ we may likewise for every $n \geq 1$ find $\psi''_n \in \mathfrak{X}$ such that 
 \begin{equation}d\left(\psi''_n(a+t'), x_2\right)<\frac{1}{n},\qquad \min_{t \in [a+t',a+t'+r]} d\left(\psi''_n(t), x_1\right)<\frac{1}{n}\end{equation}
and take an accumulation point $\psi'' \in \mathfrak{X}$ satisfying $\psi''(a+t')=x_2$, $\psi''(a+t'+t'')=x_1$ for some $t'' \in (0,r]$, and  $\psi''(t) \in \overline{B_\delta(x_1)}$ for all $t \in [a+t', a+t'+t'']$. By the switching axiom, the function $\psi$ that coincides with $\psi'$ on $[0,a+t']$ and with $\psi''$ on $[a+t',\infty)$ is an element of $\mathfrak{X}$. Letting $p \coloneq t' + t''$, it is now immediate that $\psi$ satisfies properties \eqref{eq:first_claim} -- note that $p<r$ as a consequence of \eqref{eq:non_constant}. The first claim is now proved.

Our second claim is that there exists a sequence $(\xi_n)_{n=1}^\infty$ in $\mathfrak{X}$ and an  strictly increasing and unbounded sequence $(\tau_n)_{n=0}^\infty$ with $\tau_0 = 0$ such that, for each $n$,
\begin{equation}\label{eq:second_claim}
\xi_n(\tau_n) = \xi_n(\tau_{n+1}) = x_1  \quad \text{and} \quad \max_{t \in [\tau_n,\tau_{n+1}]} d(\xi_n(t),x_1) \le \delta \, .
\end{equation}
Indeed, for each integer $k \ge 1$, let $a_k \coloneq kr$ and use the previous claim to find $\psi_k \in \mathfrak{X}$ and $p_k \in (0,r)$ such that 
\begin{equation}\label{eq:psi_k}
\psi_k(a_k) = \psi_k(a_k+p_k) = x_1 \quad \text{and} \quad \max_{t \in [a_k,a_k+p_k]} d(\psi_k(t),x_1) \le \delta \, .
\end{equation}
Let $\ell_k \coloneq \lfloor r / p_k \rfloor$; this is a positive integer.
Note that $r/2 < \ell_k p_k \le r$.
Now define a real sequence $(\tau_n)_{n=0}^\infty$ by the rules: $\tau_0 = 0$ and 
\begin{equation}
\tau_{n+1} - \tau_n = p_k \quad \text{if} \quad \ell_1 + \cdots + \ell_{k-1} < n \le \ell_1 + \cdots + \ell_k \, .
\end{equation}
Then the sequence $(\tau_n)$ is strictly increasing.
Since $\tau_{\ell_1+\cdots+\ell_k} =  \ell_1 p_1 + \cdots + \ell_k p_k \ge kr/2$, the sequence $(\tau_n)$ is also unbounded.  
Next, for each integer $n \ge 1$, fix $k \ge 0$ such that $\ell_1 + \cdots + \ell_{k-1} < n \le \ell_1 + \cdots + \ell_k$. Since $\tau_n \le \ell_1 p_1 + \cdots + \ell_k p_k \le kr = a_k$, we can define $\xi_n \coloneq \sigma^{a_k - \tau_n} \psi_k \in \mathfrak{X}$. In particular, $\xi_n(\tau_n+t) = \psi_k(a_k+t)$ for every $t \ge 0$. Now \eqref{eq:second_claim} follows directly from properties \eqref{eq:psi_k}, proving the second claim. 

Finally, we apply the multiple switching Lemma~\ref{le:multiple_switching} to concatenate all segments $\xi_n|_{[\tau_n,\tau_{n+1}]}$, obtaining a path $\xi \colon [0,\infty) \to Z$ in $\mathfrak{X}$ starting from $\xi(0) = x_1$ and entirely contained in the closed $\delta$-neighborhood of $x_1$. The existence of this path contradicts property \eqref{eq:non_constant}. 
This completes the proof in the case where $Z$ is homeomorphic to a compact subset of $S^2$.
\end{proof}

\begin{proof}[Proof of Theorem~\ref{th:stronger} for subsets of the projective plane]
Let $\mathfrak{X}$ be a switched Bebutov shift with values in a compact set $Z\subseteq \RP^2$, let $\pi \colon S^2 \to \RP^2$ be the standard twofold covering map, and define $\Pi \colon C([0,\infty), S^2) \to C([0,\infty), \RP^2)$ by $\Pi(\hat\psi)\coloneq \pi \circ \hat\psi$. A straightforward verification using the path-lifting lemma shows that $\Pi$ is a twofold covering map, so in particular it is a proper map and therefore  $\hat{\mathfrak{X}}\coloneq \Pi^{-1}\mathfrak{X}\subseteq C([0,\infty),S^2)$ is  compact. Trivially $\hat{\mathfrak{X}}$ is shift-invariant and another easy application of the path-lifting lemma shows that $\hat{\mathfrak{X}}$ inherits the switching property from $\mathfrak{X}$, so $\hat{\mathfrak{X}}$ is a switched Bebutov shift. Given a recurrent element $\phi \in \mathfrak{X}$, choose $\hat\phi \in \Pi^{-1}\phi$ arbitrarily. We claim that $\hat\phi$ is recurrent. Indeed, if $(t_n)_{n=1}^\infty$ tends to infinity and satisfies $\lim_{n \to \infty}\sigma^{t_n}\phi=\phi$ then by compactness it has a subsequence $(t_{n_k})_{k=1}^\infty$ such that $(\sigma^{t_{n_k}}\hat\phi)$ is convergent. The limit of this subsequence necessarily projects to $\phi$ and hence is either equal to $\hat\phi$ itself or to the function $-\hat\phi$ whose values are everywhere  antipodal to those of $\hat\phi$. In the former case the claim is immediately proved; in the latter case we note that $-\hat\phi \in \omega(\hat\phi)$ implies $\hat\phi \in \omega(-\hat\phi)\subseteq \omega(\hat\phi)$ and the claim is again proved. Applying the already-proved first case of Theorem~\ref{th:stronger} to $\hat\phi$ we obtain a preperiodic function $\hat\chi \in \hat{\mathfrak{X}}$ such that $\hat\chi(0)=\hat\phi(0)$. Clearly $\chi\coloneq \Pi (\hat\chi)$ is preperiodic and satisfies $\chi(0)=\phi(0)$ as required. If additionally $\phi$ is injective then clearly so is $\hat\phi$, hence $\hat\chi$ is constant and so too is $\chi$. The proof of Theorem~\ref{th:stronger} is complete.
\end{proof}

\subsection{Some Questions}\label{ss:Bebutov-questions}

Our motivation for Theorem~\ref{th:stronger} comes from the applications that we will discuss in the next section. Nevertheless, the theorem is interesting in its own right, and several natural questions arise. For instance, \emph{can the preperiodic element $\chi$ in Theorem~\ref{th:stronger} always be chosen to be periodic?}

With the aim of broadening the study of Bebutov shifts beyond the switching property, it would be worthwhile to investigate the class of all minimal Bebutov shifts $\mathfrak{X}$ in $C([0,\infty),S^2)$ (or in $C(\mathbb{R},S^2)$) that are \emph{self-avoiding}, meaning they consist entirely of injective trajectories. Examples of such shifts were seen in Remark~\ref{re:Plykin}. Slightly more intricate examples can be constructed by introducing bifurcations. Bearing in mind traditional classification problems in dynamics \cite{Foreman2020}, we pose the following question: \emph{Can we classify all self-avoiding minimal Bebutov shifts on the sphere, up to topological conjugacy?}

\section{Applications}\label{se:app}

\subsection{A Poincar\'e-Bendixson theorem for semiflows}

An extensive literature exists generalising the original Poincar\'e-Bendixson Theorem of \cite{Be01,Po80} and this topic is the subject of the survey article \cite{Ci12}. The following statement appears nonetheless to be novel in several respects:
\begin{theorem}\label{th:pbsemi}
Let $Z$ be a topological space which is homeomorphic to a compact subset of either $S^2$ or $\RP^2$, and let $\Phi \colon Z \times [0,\infty) \to Z$ be a continuous semiflow. If $z \in Z$ is recurrent with respect to $\Phi$ then it is either a fixed point or a periodic orbit. 
\end{theorem}

\begin{remark} 
The earliest extensions of the Poincar\'e-Bendixson theorem (such as the theorem of A.J.~Schwartz in \cite{Sc63}, which considers the problem on arbitrary compact $2$-manifolds) operate in the context of differentiable flows which arise from an underlying vector field which is itself assumed at least to be $C^1$. In the context of general continuous flows, a number of results treat the case in which $Z$ is replaced by an open subset of $S^2$: these include for example a theorem of P.~Seibert and P.~Tulley\footnote{Although the precise hypotheses of Seibert and Tulley's result are not explicitly stated, their argument relies on the fact that every point along a recurrent trajectory of the flow must have a neighbourhood which is a topological disc, and also that the images and preimages of these topological discs under the flow must themselves be topological discs. In particular the flow must be defined on an open neighbourhood of the recurrent trajectory being studied, and also cannot be assumed only to be a semiflow.} in \cite{SeTu67} which treats flows on open subsets of $\R^2$, and the work of O.~H\'ajek in \cite{Ha68} which applies to flows defined on $2$-manifolds for which a suitable analogue of the Jordan curve theorem is satisfied. In the context of semiflows, the special case of Theorem~\ref{th:pbsemi} in which $Z$ is equal to either $\R^2$ or $S^2$ was proved by K.~Ciesielski in \cite{Ci94}. On the other hand the case of semiflows on proper subsets of $S^2$ other than the plane appears  not to have been treated until the present work, and more broadly we are aware of no antecedent works which treat either flows or semiflows on subsets of $S^2$ or $\RP^2$ with empty interior.
\end{remark}

\begin{proof}[Proof of Theorem~\ref{th:pbsemi}]
For every $z \in Z$ define a function $\phi_z \in C([0,\infty), Z)$ by $\phi_z(t)\coloneq \Phi^tz$ for all $t \geq 0$. Let $\mathfrak{X}\coloneq \{\phi_z \colon z \in Z\}\subseteq C([0,\infty), Z)$ and define $\Pi \colon Z \to \mathfrak{X}$ by $\Pi(z)\coloneq \phi_z$. It is clear that $\Pi$ is bijective. We claim  that $\Pi$ is also continuous. Indeed, to demonstrate this it is sufficient to show that if $(z_n)$ is a sequence in $Z$ which converges to a limit $z$ then $\lim_{n \to \infty} \phi_{z_n} = \phi_z$ with respect to the compact-open topology; but if $T \geq 0$ is arbitrary then the uniform continuity of $\Phi$ on $Z \times [0,T]$ guarantees that $\phi_{z_n} \to \phi$ uniformly on $[0,T]$. Since $T$ was arbitrary the required convergence follows. We immediately deduce from the continuity of $\Pi$ that $\mathfrak{X}$ is compact, and $\mathfrak{X}$ is clearly also invariant with respect to the shift semiflow $(\sigma^\tau)_{\tau \geq 0}$ since $\sigma^\tau \phi_z = \phi_{\Phi^\tau z} \in \mathfrak{X}$ for every $z \in Z$ and $\tau \geq 0$. If $\phi_{z_1}, \phi_{z_2} \in \mathfrak{X}$ satisfy $\phi_{z_1}(\tau)=\phi_{z_2}(\tau)$ for some $\tau \geq 0$ then the function $\phi \in C([0,\infty),Z)$ defined by $\phi(t)\coloneq \phi_{z_1}(t)$ for all $t \leq \tau$ and by $\phi(t)\coloneq \phi_{z_2}(t)$ for all $t \geq \tau$ is equal to $\phi_{z_1}$ and in particular belongs to $\mathfrak{X}$. We have shown in particular that $\mathfrak{X}$ is a switched Bebutov shift.
Since $\Pi$ is a continuous bijection between compact Hausdorff spaces it is a homeomorphism, and we clearly have $\Pi \circ \Phi^t = \sigma^t \circ \Pi$ for every $t \geq 0$. In particular the semiflow $(\Phi^t)_{t \geq 0}$ on $Z$ is topologically conjugate to the shift semiflow $(\sigma^t)_{t\geq 0}$ on $\mathfrak{X}$.

Now suppose that $z \in Z$ is recurrent; then by topological conjugacy, $\phi_z \in \mathfrak{X}$ is recurrent with respect to the shift semiflow on $\mathfrak{X}$. It follows by Theorem~\ref{th:stronger} that there exists a preperiodic function $\chi=\phi_w \in \mathfrak{X}$ such that $\phi_w(0)=\phi_z(0)$; but this implies that $z=w$, so $\phi_z=\phi_w$ and therefore $\phi_z$ is preperiodic with respect to the shift semiflow on $\mathfrak{X}$. It follows by topological conjugacy that $z$ is preperiodic. Since a preperiodic recurrent point must be periodic, the proof is complete.
\end{proof}

\subsection{Periodic limit trajectories for switched continuous vector fields}

The next consequence of Theorem~\ref{th:stronger} which we exhibit is a version of the Poincar\'e-Bendixson theorem for switched continuous vector fields on the $2$-sphere and the projective plane. While this is certainly the least interesting of the applications presented here, we mention it in passing since it may be derived straightforwardly from a technical result which we will require when proving the more substantial results to follow.

\begin{theorem}\label{th:cts}
Let $M$ be either $S^2$ or $\RP^2$, let $Z\subseteq M$ be compact and for $i=1,\ldots,N$ let $\Psi_i \colon Z \to TM$ be a continuous function such that $\Psi_i(z)\in T_zM$ for every $z \in Z$. Suppose that there exists an absolutely continuous function $x \colon [0,\infty) \to Z$ which satisfies $\dot{x}(t) = \sum_{i=1}^N u_i(t) \Psi_i(x(t))$ almost everywhere for some Lebesgue measurable functions $u_1,\ldots,u_N \colon [0,\infty) \to [0,1]$ such that $\sum_{i=1}^N u_i(t)=1$ a.e. Then there exist periodic measurable functions $v_1,\ldots,v_N \colon [0,\infty) \to [0,1]$ satisfying $\sum_{i=1}^N v_i(t)=1$ a.e, and a periodic absolutely continuous function $y$ from $[0,\infty)$ to the $\omega$-limit set of $x$, such that $\dot{y}(t)=\sum_{i=1}^N v_i(t)\Psi_i(y(t))$ a.e. \end{theorem}

\begin{remark}
The un-switched case $N=1$ of Theorem \ref{th:cts} is rarely stated in the literature without further assumptions such as the unique integrability of the vector field, but it can be derived without difficulty from the results of \cite[{\S}VIII.4]{Ha64}.
\end{remark}

\begin{proposition}\label{pr:aff}
Let $Z\subset \R^d$ be compact and let  $\Psi_1,\ldots,\Psi_N \colon  Z \to \R^d$ be continuous. Let $\mathfrak{X}_Z$ denote the set of all absolutely continuous functions $x \colon [0,\infty) \to Z$ which solve a Carath\'eodory differential equation of the form
\begin{equation}\label{eq:ft}\dot{x}(t) = \sum_{i=1}^N u_i(t) \Psi_i(x(t))\text{ a.e.} \end{equation}
where $u_1,\ldots,u_N \colon [0,\infty) \to [0,1]$ are Lebesgue measurable functions such that $\sum_{i=1}^N u_i(t)=1$ a.e. Then $\mathfrak{X}_Z$ is a switched Bebutov shift.
\end{proposition}
\begin{proof}
If $\mathfrak{X}_Z$ is empty then the result holds trivially, so we suppose otherwise. The shift-invariance axiom of Definition~\ref{de:Bebutov} is elementary: if $x \in \mathfrak{X}_Z$ satisfies \eqref{eq:ft} then for every $\tau \geq 0$ the function $\sigma^\tau x \colon [0,\infty) \to Z$ defined by $(\sigma^\tau x) (t)\coloneq x(t+\tau)$ solves the same equation with the functions $v_1,\ldots,v_N \colon [0,\infty) \to [0,1]$ defined by $v_i(t)\coloneq u_i(t+\tau)$  in place of the functions $u_1,\ldots,u_N$. The switching axiom of Definition~\ref{de:switched} is also straightforward: given $x^{(1)}, x^{(2)} \in \mathfrak{X}_Z$  such that $x^{(1)}(\tau)=x^{(2)}(\tau)$ for some $\tau \geq 0$, suppose that $\dot{x}^{(1)}(t)=\sum_{i=1}^N u^{(1)}_i(t)\Psi_i(x(t))$ and $\dot{x}^{(2)}(t)=\sum_{i=1}^N u_i^{(2)}(t)\Psi_i(x(t))$ a.e.\ where $\sum_{i=1}^N u_i^{(j)}(t)=1$ a.e.\ for $j=1,2$. Define $x(t)\coloneq x^{(1)}(t)$ for $0 \leq t \leq \tau$ and $x(t)\coloneq x^{(2)}(t)$ for $t\geq \tau$, for each $i=1,\ldots,N$ define $u_i(t)\coloneq u^{(1)}_i(t)$ for $0 \leq t \leq \tau$ and $u_i(t)\coloneq u^{(2)}_i(t)$ for $t \geq \tau$, and observe that \eqref{eq:ft} is satisfied so that $x \in \mathfrak{X}_Z$ as required.

It remains to show that $\mathfrak{X}_Z$ is compact as a subspace of $C([0,\infty), Z)$ in the compact-open topology. Since the latter is metrisable it is sufficient to demonstrate sequential compactness. Define $K\coloneq \max_{1 \leq i \leq N} \max_{v \in Z}\|\Psi_i(v)\|$ and observe that for every $x \in \mathfrak{X}_Z$ and  $t_2 \geq t_1 \geq 0$ we have
\begin{equation}\left\|x(t_2)-x(t_1)\right\| = \left\|\int_{t_1}^{t_2} \dot{x}(s)ds\right\| \leq  K|t_2-t_1|\end{equation}
which implies that $\mathfrak{X}_Z$ is uniformly equicontinuous. In view of the Arzel\`a-Ascoli theorem, to demonstrate sequential compactness we must show that $\mathfrak{X}_Z$  is also closed in $C([0,\infty), Z)$. To this end, suppose that $(x^{(n)})_{n=1}^\infty$ is a sequence of elements of $\mathfrak{X}_Z$ which converges to a limit $x \in C([0,\infty), Z)$, and let us demonstrate that $x \in \mathfrak{X}_Z$.

 For every $n \geq 1$  let $u_1^{(n)},\ldots,u_N^{(n)} \colon [0,\infty) \to [0,1]$ be measurable functions such that the equations $\dot{x}^{(n)}(t) = \sum_{i=1}^N u_i^{(n)}(t) \Psi_i(x^{(n)}(t))$ and $\sum_{i=1}^N u^{(n)}_i(t)=1$ are satisfied a.e.\ for every $n \geq 1$.  We recall that the weak\nobreakdashes-* topology on $L^\infty([0,\infty))$ is defined to be the finest topology such that the map $u \mapsto \int_0^\infty u(t)f(t)dt$ is continuous for every $f \in L^1([0,\infty))$, and we recall also that with respect to the weak\nobreakdashes-* topology on $L^\infty([0,\infty))$ the set of measurable functions $[0,\infty) \to [0,1]$ is compact and metrisable (see e.g \cite[Theorems 3.15 and 3.16]{Ru91}).
Using sequential compactness of the set of measurable functions $[0,\infty) \to [0,1]$ in this topology, and by passing to a subsequence if necessary, we suppose that every $(u^{(n)}_i)_{n=1}^\infty$ converges in the weak\nobreakdashes-* topology as $n \to \infty$ to a measurable function $u_i \colon [0,\infty) \to [0,1]$. By integrating against characteristic functions of intervals and appealing to weak\nobreakdashes-* convergence it is clear that the equation $\sum_{i=1}^N u_i(t)=1$ a.e.\ is satisfied. 

To prove that $x \in \mathfrak{X}_Z$ we will show that $\dot{x}(t)=\sum_{i=1}^N u_i(t)\Psi_i(x(t))$ a.e. 
Via the Lebesgue differentation theorem it suffices to show that for every $T>0$, 
\begin{equation}\label{eq:wk-goal}\int_{0}^{T} \dot{x}(t) dt = \int_{0}^{T} \sum_{i=1}^N u_i(t)\Psi_i(x(t)) dt.\end{equation}
Fix $T>0$. 
For every $n \geq 1$ we may write
\begin{align}\label{eq:wk-diff}
\MoveEqLeft[4]{\int_{0}^{T} \dot{x}(t) dt - \int_{0}^{T} \sum_{i=1}^N u_i(t)\Psi_i(x(t)) dt} &\\
&= \int_{0}^{T} \dot{x}(t) dt - \int_{0}^{T} \dot{x}^{(n)}(t) dt \label{eq:wk1} \\
 &\qquad + \int_{0}^{T} \dot{x}^{(n)}(t)dt - \int_0^T \sum_{i=1}^N u_i^{(n)}(t) \Psi_i(x^{(n)}(t))dt \label{eq:wk2} \\
 &\qquad +  \int_0^T  \sum_{i=1}^N u_i^{(n)}(t)\left(\Psi_i\big(x^{(n)}(t)\big)-\Psi_i(x(t))\right)dt \label{eq:wk3} \\
 &\qquad + \int_0^T  \sum_{i=1}^N \left(u_i^{(n)}(t)-u_i(t)\right)\Psi_i\left(x(t)\right)dt \label{eq:wk4} \, .
\end{align}
The expression \eqref{eq:wk1} reduces to $x(T)-x^{(n)}(T)-x(0)+x^{(n)}(0)$ which obviously tends to zero as $n \to \infty$ since $x^{(n)} \to x$ in $\mathfrak{X}_Z$, and the term \eqref{eq:wk2} equals zero by definition of $x^{(n)}$. Since each $\Psi_i$ is uniformly continuous on $Z$, and since $x^{(n)} \to x$ uniformly on $[0,T]$, we have $\lim_{n \to \infty} \sup_{t \in [0,T]} \|\Psi_i(x^{(n)}(t))-\Psi_i(x(t))\|=0$ for every $i=1,\ldots,N$ so that the quantity \eqref{eq:wk3} tends to zero as $n \to \infty$.
Furthermore, the quantity \eqref{eq:wk4} also tends to zero, by appealing to the weak\nobreakdashes-* convergence of each $u_i^{(n)}$ to the corresponding $u_i$. By choosing $n$ sufficiently large we thus obtain an arbitrarily small bound for the difference \eqref{eq:wk-diff}.
We conclude that \eqref{eq:wk-goal} is satisfied for every $T>0$, so $\dot{x}(t)=\sum_{i=1}^N u_i(t)\Psi_i(x(t))$ a.e.\ and therefore $x \in \mathfrak{X}_Z$ as required. The proposition is proved.
\end{proof}

\begin{proof}[Proof of Theorem~\ref{th:cts}]
Suppose that $M=S^2$. If the $\omega$-limit set of $x$ is the whole of $S^2$ then in particular $Z=S^2$. It is a well-known consequence of the Poincar\'e-Hopf theorem that a continuous vector field on $S^2$ must vanish at some point, so in particular every $\Psi_i$ must have a fixed point somewhere in $S^2$ and there therefore exists a constant solution $y$ whose constant value trivially belongs to the $\omega$-limit set of $x$. Now suppose instead  that the $\omega$-limit set of $x$ is a proper subset of $S^2$. In this case, for every $T \geq 0$ define $Z_T\coloneq \overline{\{x(t)\colon t \geq T\}}$ and observe that for all large enough $T$ --- say, for all $T \geq T_0$ --- the sets $Z_T$ are all contained in the complement of some particular closed disc in $S^2$, and this complement is necessarily diffeomorphic to $\R^2$. By  Proposition~\ref{pr:aff} and the Birkhoff recurrence theorem the nonempty compact set $\mathfrak{X}_{\bigcap_{T\geq T_0}Z_T}=\bigcap_{T \geq T_0}\mathfrak{X}_{Z_T}$ contains a recurrent point with respect to the shift semiflow, so by Theorem~\ref{th:stronger} this set contains a periodic trajectory $y$. This concludes the argument in those cases where $Z$ is a subset of $S^2$.

If instead $M=\RP^2$ we apply a lifting argument. Let $\pi \colon S^2 \to \RP^2$ denote the standard twofold covering map and let $\hat{Z}\coloneq \pi^{-1}Z$. By the path-lifting lemma the trajectory $x \colon [0,\infty) \to Z$ lifts to a continuous function $\hat{x} \colon [0,\infty) \to \hat{Z}$ which is easily seen to be Lipschitz continuous and to satisfy a differential equation of the required form. We may now apply the first case to deduce the existence of a periodic trajectory in the $\omega$-limit set of $\hat{x}$ and this is projected by $\pi$ to the desired trajectory in the $\omega$-limit set of $x$.
\end{proof}

\subsection{Stability of switched homogeneous systems in up to three real dimensions}

If $\Psi_1,\ldots,\Psi_N \colon \R^d \to \R^d$ are continuous homogeneous functions then   the  notions of global uniform exponentially stability, Lyapunov stability and periodic asymptotic stability may be defined for the associated switched homogeneous system in precisely the same way as was described in the introduction for switched linear systems. The following result establishes a new relationship between these notions of stability, in the spirit of works such as \cite{An99,AnDeSo09,DaMa99,SoWa96}.

\begin{theorem}\label{th:homog}
Let $1 \leq d \leq 3$ and let $\Psi_1,\ldots,\Psi_N \colon \R^d \to \R^d$ be homogeneous and uniformly Lipschitz continuous. If the switched system on $\R^d$ defined by $\Psi_1,\ldots,\Psi_N$ is Lyapunov stable and periodically asymptotically stable, then it is globally uniformly exponentially stable.
\end{theorem}
In dimension four and higher the corresponding implication is false: for an example see \cite{Mo25}. 

The proof of Theorem~\ref{th:homog} results from the combination of Theorem~\ref{th:stronger} and Proposition~\ref{pr:aff} with two additional results which follow. Both of these results apply in arbitrary dimensions, and indeed the  only point in the proof of Theorem~\ref{th:homog} at which the hypothesis $d \leq 3$ is applied is when appealing to Theorem~\ref{th:stronger}. In this subsection, when $\Psi_1,\ldots,\Psi_N$ are understood we use the notation $\mathfrak{X}_Z$ to denote the set of all trajectories of the switched system defined by $(\Psi_1,\ldots,\Psi_N)$ which take values only in the set $Z\subseteq \R^d$.

The following Fenichel-type lemma is a special case of \cite[Theorem 2]{AnDeSo09}, where it is deduced from results presented in \cite{SoWa96}. In practice we will require this result only in the case where the switched system defined by $\Psi_1,\ldots,\Psi_N$ is Lyapunov stable, and in this case a trajectory which accumulates at the origin necessarily converges to the origin. In that form the lemma also occurs in the unpublished manuscript~\cite{An99}. 

\begin{lemma}\label{le:fenichel}
Let $\Psi_1,\ldots,\Psi_N \colon \R^d \to \R^d$ be homogeneous and uniformly Lipschitz continuous. If every trajectory of the switched system on $\R^d$ defined by $\Psi_1,\ldots,\Psi_N$ accumulates at the origin, then the system is globally uniformly exponentially stable.
\end{lemma}

\begin{proposition}\label{pr:minimal}
Let $\Psi_1,\ldots,\Psi_N \colon \R^d \to \R^d$ be homogeneous and uniformly Lipschitz continuous and suppose that the switched system on $\R^d$ defined by $\Psi_1,\ldots,\Psi_N$ is Lyapunov stable. Suppose $Z \subset \R^d\setminus\{0\}$ is a compact set such that $\mathfrak{X}_Z \neq \emptyset$ and which has no proper subset with the same properties. Then $Z$ is homeomorphic to a subset of $S^{d-1}$ via the map $v \mapsto v/\|v\|$.
\end{proposition}
\begin{proof}
We will prove that $Z$ intersects each positive ray $\{\lambda v \colon \lambda \in [0,\infty)\}$ in at most one point, which implies that the continuous function $\pi \colon Z \to S^{d-1}$ defined by $\pi(v)\coloneq \|v\|^{-1}v$ must be injective and hence (by the compactness of $Z$) that $\pi$ is a homeomorphism onto its image. Using Lyapunov stability, let $C>0$ be a constant such that $\sup_{t \geq 0}\|x(t)\|\leq C\|x(0)\|$ for every trajectory $x$. We suppose for a contradiction that $Z$ intersects a positive ray in two points, in which case there exist $v \in Z$ and $\lambda>1$ such that $\lambda v \in Z$. Throughout the proof we fix such a $v$ and $\lambda$, and also fix $n\geq 1$ large enough that $\lambda^n>2C$. To obtain a contradiction we will construct a trajectory $z$ and time $T>0$ such that $\|z(T)\|\geq(\lambda^n/2)\|z(0)\|> C\|z(0)\|$. To construct this trajectory we will begin by identifying segments of trajectories which originate at $v$ and then closely approach $\lambda v$, and by concatenating segments of these trajectories' switching laws derive the switching law which will generate $z$. The second part of the argument is complicated by a need to work backwards in time, beginning with the final segment of the switching law for $z$, constructing the earlier segments inductively in reverse order, and verifying the desired property of $z$ via a reverse induction.

To begin the construction we  claim that there exists a trajectory $x \in \mathfrak{X}_Z$ such that $x(0)=v$ and $\liminf_{t \to \infty} \|x(t)-\lambda v\|=0$. Indeed, let $y \in \mathfrak{X}_Z$ be arbitrary. If $Z'\coloneq \overline{\{y(t) \colon t \in [0,\infty)\}}\subseteq Z$ is not equal to $Z$ then it is a nonempty compact proper subset of $Z$ with the property that $\mathfrak{X}_{Z'}$ is nonempty, which by definition is impossible. It follows that $Z'=Z$ and hence necessarily $\lim_{k \to \infty} y(t_k)=v$ for some sequence of times $(t_k)_{k=1}^\infty$. Choose $x \in \mathfrak{X}_Z$ to be an accumulation point of $(\sigma^{t_k}y)_{k=1}^\infty$ as $k \to \infty$ and observe that $x(0)=v$. By similar reasoning to before the set $\overline{\{x(t) \colon t \in [T,\infty)\}}$ must equal $Z$ for every $T \geq 0$ or else the definition of $Z$ is contradicted, and it follows that $\lambda v \in \overline{\{x(t) \colon t \in [T,\infty)\}}$ for every $T \geq 0$ as required to prove the claim.

Define $\delta\coloneq \lambda^n\|v\|/2n>0$ and choose $K>0$ such that $\max_{1 \leq i \leq N}\|\Psi_i(w_1)-\Psi_i(w_2)\| \leq K\|w_1-w_2\|$ for all $w_1,w_2 \in \R^d$. 
Let $x \in \mathfrak{X}_Z$ be the trajectory given by the previous claim and let $u_1,\ldots,u_N \colon [0,\infty) \to [0,1]$ be measurable switching laws generating the trajectory $x$. Choose a time $t_n>0$ such that $\|x(t_n)-\lambda v\|<\lambda^{1-n}\delta$, and then inductively choose times $t_{n-1}, t_{n-2},\ldots \ldots,t_2, t_1>0$ such that 
\begin{equation}\label{eq:tea-choice}\left\|x(t_k)-\lambda v\right\|<\lambda^{1-k}e^{-K\sum_{j=k+1}^n t_j} \, \delta\end{equation}
for every $k=1,\ldots,n$.  
Write $T_k \coloneq \sum_{j=1}^k t_j$. 
Now let $\hat{u}_1,\ldots,\hat{u}_N \colon [0,\infty) \to [0,1]$ be measurable functions such that $\hat{u}_i(T_{k-1} +s)=u_i(s)$ for a.e.\ $s \in [0,t_k]$  for every $k=1,\ldots,n$, and such that $\sum_{i=1}^N \hat{u}_i = 1$ a.e. Let $z$ be the trajectory corresponding to the switching laws $\hat{u}_1,\ldots,\hat{u}_N$ with initial condition $z(0)=v$. We claim that for every $k =1,\ldots,n$,
\begin{equation}\label{eq:clammy}\left\|z\left(T_k\right) - \lambda^k v\right\| \leq e^{Kt_k} \left\|z\left(T_{k-1}\right) - \lambda^{k-1} v\right\| + e^{-K\sum_{j=k+1}^n t_j} \, \delta.\end{equation}
The case $k=1$ follows trivially from \eqref{eq:tea-choice} together with the fact that $z(t_1)=x(t_1)$, so suppose instead that $k$ lies in the range $2,\ldots,n$. For every $\tau \in [0, t_k]$,  
\begin{align}\MoveEqLeft[5]{\left\|\int_0^\tau \left( \dot{z}\left(T_{k-1} + s\right) - \lambda^{k-1}\dot{x}(s)\right) ds\right\|}&\\\nonumber
 &= \left\|\int_0^\tau\sum_{i=1}^N u_i(s) \left( \Psi_i\left(z\left(T_{k-1} + s\right)\right) - \Psi_i\left( \lambda^{k-1}x(s)\right)\right) ds\right\|\\\nonumber
&\leq K\int_0^\tau \left\|z\left(T_{k-1}+s\right)  - \lambda^{k-1}x(s)\right\|ds\end{align}
where we have used the homogeneity of $\Psi_1,\ldots,\Psi_N$. We in particular have
\begin{multline}\left\|z\left(T_{k-1}+\tau\right) - \lambda^{k-1}x(\tau)\right\| \leq \left\|z\left(T_{k-1}\right) - \lambda^{k-1}x(0)\right\| \\ 
 +K\int_0^\tau \left\|z\left(T_{k-1}+s\right)  - \lambda^{k-1}x(s)\right\|ds \end{multline}
for all $\tau \in [0,t_k]$. It follows by Gr\"onwall's inequality that 
\begin{equation}\left\|z\left(T_k\right) - \lambda^{k-1}x(t_k) \right\| \leq e^{Kt_k} \left\|z\left(T_{k-1}\right) - \lambda^{k-1}x(0)\right\|\end{equation}
and the claimed inequality \eqref{eq:clammy} follows using $x(0)=v$ together with the inequality~\eqref{eq:tea-choice}. The inequality
\begin{equation}\left\|z\left(T_n\right) - \lambda^n v\right\| \leq e^{K\sum_{j=k}^n t_j} \left\|z\left(T_{k-1}\right) - \lambda^{k-1} v\right\| + (n+1-k)\delta\end{equation}
for all $k=1,\ldots,n$ follows from \eqref{eq:clammy} by a downward induction starting with $k=n$. In the  case $k=1$ the first term on the right-hand side evaluates to zero, yielding
\begin{equation}\left\|z\left(T_n\right) - \lambda^n v\right\| \leq n\delta = \frac{1}{2} \left\|\lambda^n v\right\|\end{equation}
and therefore
\begin{equation}\left\|z\left(T_n\right)\right\|\geq\frac{\lambda^n\|v\|}{2} >C\|v\|=C\left\|z(0)\right\|\end{equation}
which contradicts the definition of $C$. We conclude that $Z$ cannot intersect a positive ray in two distinct points and the result follows.
\end{proof}

\begin{proof}[Proof of Theorem~\ref{th:homog}]
Let $d \leq 3$ and suppose that the switched system on $\R^d$ defined by $\Psi_1,\ldots,\Psi_N$ is Lyapunov stable but is \emph{not} globally uniformly exponentially stable. To prove the theorem we will show that this system is not periodically asymptotically stable. 

It follows from Lemma~\ref{le:fenichel} that there exists a trajectory $x$ which does not accumulate at the origin, and by Lyapunov stability this trajectory is also bounded. Define $Y\coloneq \overline{\{x(t) \colon t \geq 0\}}$ and observe that $Y$ is a compact subset of $\R^d\setminus \{0\}$ with the property that $\mathfrak{X}_Y$ is nonempty. Let $\mathscr{Z}$ denote the set of all compact, nonempty sets $Z\subseteq Y$ with the property that $\mathfrak{X}_Z$ is nonempty. Clearly $\mathscr{Z}$ is nonempty since it contains the element $Y$. 

We claim that $ \mathscr{Z}$ has a minimal element with respect to inclusion. By Zorn's lemma it is sufficient to show that if $(Z_\alpha)_{\alpha \in A}$ is a family of elements of $\mathscr{Z}$ which is totally ordered with respect to inclusion, then there exists $Z_0 \in \mathscr{Z}$ which is a subset of every $Z_\alpha$. Given such a family $(Z_\alpha)_{\alpha \in A}$ define $Z_0\coloneq \bigcap_{\alpha \in A}Z_\alpha\subseteq Y$ which is compact and nonempty by standard results of general topology. We then have $\mathfrak{X}_{Z_0}=\mathfrak{X}_{\bigcap_{\alpha \in A} Z_\alpha} = \bigcap_{\alpha \in A} \mathfrak{X}_{Z_\alpha}  \neq \emptyset$ by applying the same principle to the totally ordered family of nonempty compact sets $(\mathfrak{X}_{Z_\alpha})_{\alpha \in A}$, each of which is a closed subset of the compact set $\mathfrak{X}_Y$. In particular $Z_0$ is the required element of $\mathscr{Z}$ and Zorn's lemma applies, proving the claim.

Now let $Z\subseteq Y \subset \mathbb{R}^d\setminus \{0\}$ be as given by the claim, and observe that by Proposition~\ref{pr:minimal} the set $Z$ is homeomorphic to a subset of $S^{d-1}\subseteq S^2$. By Proposition~\ref{pr:aff} the set $\mathfrak{X}_Z$ is a switched Bebutov shift, and by the Birkhoff recurrence theorem there exists an element of $\mathfrak{X}_Z$ which is recurrent with respect to the shift semiflow $(\sigma^\tau)_{\tau \geq 0}$ on $\mathfrak{X}_Z$, so by Theorem~\ref{th:stronger} there exists a periodic trajectory $x \in \mathfrak{X}_Z$. Clearly the existence of a switching law which generates the nonzero periodic trajectory $x$ implies the existence of a \emph{periodic} switching law generating the same trajectory.  It follows that the switched system defined by $\Psi_1,\ldots,\Psi_N$ is not periodically asymptotically stable, and the proof is complete. \end{proof}

\subsection{Stability of switched linear systems in up to three real dimensions}
We may now prove unconditionally that periodic asymptotic stability implies global uniform exponential stability for real linear switched systems in up to three dimensions, removing a technical hypothesis from a theorem of Chitour, Gaye and Mason (see \cite{ChGaMa15}) and answering Question~\ref{qu:siamreview} positively in that dimension range. In the case where the system is Lyapunov stable this implication follows directly from Theorem~\ref{th:homog}, so the proof consists principally of an analysis of those cases in which Lyapunov stability does not hold. In this we will be aided by an earlier work of Chitour, Mason and Sigalotti on marginally unstable switched systems \cite{ChMaSi12}.  

If $\mathsf{A}=(A_1,\ldots,A_N)$ is a tuple of $d \times d$ real matrices, we define the corresponding \emph{uniform exponential rate} is defined as 
\begin{equation} \label{eq:UER}
\Lambda(\mathsf{A})\coloneq \lim_{t \to \infty}  \frac{1}{t} \sup_{x(0)\neq 0}\log \left(\frac{\|x(t)\|}{\|x(0)\|}\right)
\end{equation}
where the supremum is over all nonzero trajectories $x$ of the linear switched system defined by $\A$. The existence of the limit is guaranteed by subadditivity: see for example \cite[\S1.3.2]{ChMaSi25} for a detailed analysis.
The quantity $e^{\Lambda(A)}$ is a continuous-time analog of the joint spectral radius \eqref{eq:JSR}.

\begin{theorem}\label{th:linear}
Let $\mathsf{A}=(A_1,\ldots,A_N)$ be a tuple of $d \times d$ real matrices, where $1 \leq d\leq 3$. If the linear switched system defined by $\mathsf{A}$ is periodically asymptotically stable, then it is globally uniformly exponentially stable.
\end{theorem}
\begin{proof} Without loss of generality we assume $d=3$ since the lower-dimensional cases can easily be embedded in the three-dimensional case. 
To prove the theorem we suppose that the switched system is not globally uniformly exponentially stable, with the aim of deducing that it cannot be periodically asymptotically stable. Clearly if the uniform exponential rate $\Lambda(\A)$ is negative then the switched system defined by $\mathsf{A}$ is globally uniformly exponentially stable, so under our hypotheses $\Lambda(\A)$ is necessarily non-negative. 

Suppose first that $\Lambda(\A)$ is precisely zero. If the switched system defined by $\mathsf{A}$ is Lyapunov stable then the result follows by Theorem~\ref{th:homog}, so we suppose that the system is \emph{not} Lyapunov stable. In this case the switched system defined by $\mathsf{A}$ is \emph{marginally unstable} in the sense of Z.~Sun's article \cite{Su08}. By a general theorem of Chitour, Mason and Sigalotti \cite[Theorem~2]{ChMaSi12} this implies that one of the following three possibilities holds:

\emph{Case~1.} There exists an invertible $3\times 3$ real matrix $X$ such that for every $i=1,\ldots,N$ we may write
\begin{equation}A_i= X \begin{pmatrix} B_i & v_i \\ 0& c_i\end{pmatrix}X^{-1}  \end{equation}
where $B_i$ is a real $2\times 2$ matrix, $v_i$ is a real $2\times 1$ matrix and $c_i$ is a real $1\times 1$ matrix. Moreover, if we define $\mathsf{B}=(B_1,\ldots,B_N)$ and $\mathsf{C}=(c_1,\ldots,c_N)$ then $\Lambda(\mathsf{B})=\Lambda(\mathsf{C})=0$ and the linear switched system defined by $\mathsf{B}$ is Lyapunov stable.
In this case we note that by the one-dimensionality of $c_1,\ldots,c_N$, the limit $\Lambda(\mathsf{C})$ is easily seen to simply equal the maximum of the real numbers $c_1,\ldots,c_N$. Since this maximum is zero it follows that one of the matrices $X^{-1}A_iX$ has only zeros in its bottom row, and consequently the corresponding matrix $A_i$ has a zero eigenvalue. Consequently there exists a constant switching law which generates a nonzero stationary trajectory of the linear switched system defined by $\mathsf{A}$, and the system is not periodically asymptotically stable.

\emph{Case~2.} There exists an invertible $3\times 3$ real matrix $X$ such that for every $i=1,\ldots,N$ we may write
\begin{equation}A_i= X \begin{pmatrix} b_i & v_i^\mathsf{T} \\ 0& C_i\end{pmatrix}X^{-1}  \end{equation}
where $b_i$ is a real $1\times 1$ matrix, $v_i^\mathsf{T}$ is a real $1\times 2$ matrix and $C_i$ is a real $2\times 2$ matrix. Moreover, if we define $\mathsf{B}=(B_1,\ldots,B_N)$ and $\mathsf{C}=(c_1,\ldots,c_N)$ then $\Lambda(\mathsf{B})=\Lambda(\mathsf{C})=0$  and the linear switched system defined by $\mathsf{C}$ is Lyapunov stable. Similarly to the previous case, the one-dimensionality of $b_1,\ldots,b_N$ implies that one of the real numbers $b_i$ must be equal to $\Lambda(\B)$, hence one of the matrices $X^{-1}A_iX$ has a zero column, hence there exists a constant switching law which generates a nonzero stationary trajectory and the system is not periodically asymptotically stable.

\emph{Case~3.} There exists an  invertible $3\times 3$ real matrix $X$ such that for every $i=1,\ldots,N$ we may write
\begin{equation}A_i= X \begin{pmatrix} b_i & u_i & v_i \\ 0&c_i & w_i\\ 0&0&d_i\end{pmatrix}X^{-1}  \end{equation}
where each of $b_i$, $c_i$, $d_i$, $u_i$, $v_i$, $w_i$ is a real number, which we view as a $1\times 1$ real matrix. Moreover, if we define $\mathsf{B}=(b_1,\ldots,b_N)$, $\mathsf{C}=(c_1,\ldots,c_N)$ and $\mathsf{D}=(d_1,\ldots,d_N)$ then at least two of $\Lambda(\mathsf{B})$, $\Lambda(\mathsf{C})$ and $\Lambda(\mathsf{D})$ are zero.  In the same manner as the previous cases, the fact that one of these three limits is equal to zero implies that there exists $i$ such that one of $b_i$, $c_i$ or $d_i$ is zero. The corresponding matrix $X^{-1}A_iX$ is thus upper triangular with a zero diagonal entry, hence there again exists a constant switching law which generates a nonzero stationary trajectory and the system is not periodically asymptotically stable.

Summarising this analysis, we have shown that if $\Lambda(\A)=0$ and the linear switched system generated by $\mathsf{A}$ is Lyapunov stable, then it fails to be periodically asymptotically stable as a consequence of Theorem~\ref{th:homog}; but if $\Lambda(\A)=0$  and the system is \emph{not} Lyapunov stable, then it has a stationary trajectory as a consequence of \cite[Theorem 2]{ChMaSi12} and of the above case-by-case analysis, and is again not periodically asymptotically stable.

It remains to consider the case in which $\Lambda(\A)>0$. In this case, consider the linear switched system $\hat{\mathsf{A}}=(\hat{A}_1,\ldots,\hat{A}_N)$ defined by $\hat{A}_i\coloneq A_i-\Lambda(\A) \cdot I$ for $i=1,\ldots,N$. It is easily checked that $\hat{x}$ is a trajectory of the system defined by $\hat{\mathsf{A}}$ if and only if it has the form $\hat{x}(t)\equiv e^{- t\Lambda(\A)}x(t)$ for some trajectory $x$ of the system defined by $\mathsf{A}$ which is generated by an identical switching law $(u_1,\ldots,u_N)$. It follows directly from this observation that $\Lambda(\hat{\mathsf{A}})=0$. Applying the preceding analysis to $\hat{\mathsf{A}}$ it follows that there exists a periodic or stationary trajectory $\hat{x}(t)$ of the system defined by $\hat{\mathsf{A}}$ which is generated by a periodic or stationary switching law $(u_1,\ldots,u_N)$. The same switching law $(u_1,\ldots,u_N)$ thus generates a trajectory $x(t)=e^{t\Lambda(\A)} \hat{x}(t)$ of the switched system defined by $\mathsf{A}$ which escapes exponentially to infinity. In particular $\mathsf{A}$ fails to be periodically asymptotically stable. We have shown that if $\mathsf{A}$ is not globally uniformly exponentially stable then it is not periodically asymptotically stable, and the proof of the theorem is complete.
\end{proof}

\subsection{Stability of switched linear systems in two complex dimensions}
The techniques of this article also yield a proof that periodic stability implies global uniform exponential stability for certain \emph{complex} linear switched systems, this time in dimension two:
\begin{theorem}
Let $\mathsf{A}=(A_1,\ldots,A_N)$ be a tuple of $2 \times 2$ complex matrices. If the linear switched system defined by $\mathsf{A}$ is periodically asymptotically stable, then it is globally uniformly exponentially stable.
\end{theorem}
\begin{proof}
Similarly to the proof of Theorem~\ref{th:linear}, we suppose that $\A$ is not globally uniformly exponentially stable with the aim of demonstrating that it is also not periodically asymptotically stable. 
Since $\A$ is assumed not to be globally uniformly exponentially stable we have $\Lambda(\A)\geq 0$. We suppose without loss of generality that the uniform exponential rate $\Lambda(\A)$ is zero, since the case in which it is positive can be derived from the case where it is zero exactly as in the proof of Theorem~\ref{th:linear}. If any $A_i$ has an eigenvalue with non-negative real part then there trivially exists a constant switching law generating a trajectory which does not tend to zero, in which case periodic asymptotic stability is not satisfied. For the remainder of the proof we suppose instead that every eigenvalue of every $A_i$ has negative real part.

If any nonzero proper $\C$-linear subspace of $\C^2$ is simultaneously preserved by every $A_i$, then by a simultaneous change of basis we may suppose that every $A_i$ is upper triangular. By a further, diagonal change of basis we may suppose that the absolute value of every off-diagonal entry of every $A_i$ is arbitrarily small, and when these off-diagonal entries are small enough relative to the real parts of the diagonal entries, the Euclidean norm is a strict Lyapunov function for $\A$. This implies that  $\A$ is globally uniformly exponentially stable, contradicting our earlier assumption. We conclude from this that the matrices $A_i$ cannot have a nontrivial common invariant subspace. This implies that we may construct a Barabanov norm for $\A$ on $\C^2$ by following precisely the arguments of \cite{Ba88b,ChMaSi25} but using complex instead of real vector spaces. Let $Z$ denote the unit sphere of this Barabanov norm. By identifying $\C^2$ with $\R^4$ and $\A$ with an element of $M_4(\R)^N$, Proposition~\ref{pr:aff} implies that the set $\mathfrak{X}_Z$ of all $Z$-valued trajectories of the switched linear system defined by $\A$ is a switched Bebutov shift. 
By the defining property of a Barabanov norm the set $\mathfrak{X}_Z$ is nonempty. Now let $\pi \colon \C^2 \setminus \{0\} \to \CP^1$ be the continuous map which takes each nonzero vector in $\C^2$ to the one-dimensional $\C$-linear subspace which it generates, and consider the $\CP^1$-valued switched Bebutov shift $\{\pi \circ x \colon x \in \mathfrak{X}_Z\}$. Since the complex projective line $\CP^1$ is homeomorphic to $S^2$, by Theorem~\ref{th:stronger} there exists a trajectory $x \colon [0,\infty) \to Z$ of the linear switched system defined by $\A$ with the property that $\pi \circ x$ is periodic, and without loss of generality we may assume that the switching law which generates this trajectory is also periodic. The trajectory $x$ does not accumulate at zero since its range is contained in the unit sphere of the Barabanov norm, but it is generated by a periodic switching law, and thus periodic asymptotic stability is not satisfied.
\end{proof}

\subsection{Non-uniqueness of extremal trajectories for switched linear systems}\label{ss:split}

We close this article by noting a result which strongly suggests that in order to prove Theorem~\ref{th:linear} and answer the three-dimensional case of Question~\ref{qu:siamreview}, there was essentially no option other than to study some kind of multi-valued or branching flow on the sphere or projective plane (as we have undertaken in this article by the use of Bebutov shifts). More specifically, the strategy outlined in the introduction, in which one studies trajectories of a linear switched system which lie on the unit sphere of a Barabanov norm, cannot in general reduce to the analysis of an ordinary semiflow on the unit sphere of that norm.

If the uniform exponential rate $\Lambda(A)$ defined in \eqref{eq:UER} vanishes and $\A$ is Lyapunov stable, then a Barabanov norm for $\A$ is precisely a norm $\threebar{\cdot}$ on $\R^d$ with the property that every unit vector is the origin of at least one trajectory which remains in the unit sphere of $\threebar{\cdot}$ for all time. We call such a trajectory an \emph{extremal trajectory}  with respect to the Barabanov norm $\threebar{\cdot}$. If the only linear subspaces of $\R^d$ which are simultaneously preserved by every $A_1,\ldots,A_N$ are the trivial subspaces $\{0\}$ and $\R^d$, then the existence of a Barabanov norm is guaranteed (see for example \cite{Ba88b,ChMaSi25}); however this is not a necessary condition for the existence of a Barabanov norm (see \cite{ChMaSi22}).  The result of this section is the following:
\begin{theorem}\label{th:bara}
Let $\A=(A_1,\ldots,A_N)$ be a tuple of $d \times d$ real matrices satisfying $\Lambda(\A)=0$ and which admits a Barabanov norm $\threebar{\cdot}$. Suppose also that $d$ is odd and that every convex combination of the matrices $A_1,\ldots,A_N$ is nonsingular. Then there exist two distinct trajectories of the linear switched system defined by $\A$ which are extremal  with respect to $\threebar{\cdot}$ and which share the same initial state.
\end{theorem}
In order to prove Theorem~\ref{th:bara} we require the following simple topological result:
\begin{lemma}\label{le:fixed-sphere}
Let $(\Phi^t)_{t \geq 0}$ be a semiflow on a topological space $Z$ which is homeomorphic to an even-dimensional sphere $S^{2n}$. Then $(\Phi^t)$ has a fixed point.
\end{lemma}
\begin{proof}
We suppose without loss of generality that $Z=S^{2n}$. It is well known (see for example \cite[p.~367]{Mu00}) that a continuous function from $S^{2n}$ to itself either has a fixed point, or maps some point $x \in S^{2n}$ to its antipode $-x$. Since $(\Phi^t)$ is a semiflow we have $\lim_{k \to \infty} \sup_{x \in S^{2n}} \|x-\Phi^{2^{-k}}x\|=0$, so in particular if $k$ is sufficiently large then $\Phi^{2^{-k}}x \neq -x$ for every $x \in S^{2n}$. For every large enough $k$ the compact set $X_k\coloneq \{x \in S^{2n} \colon \Phi^{2^{-k}}x=x\}$ is thus nonempty. The nested intersection $\bigcap_{k=1}^\infty X_k$ is consequently nonempty and its elements are fixed points of the semiflow.
\end{proof}

\begin{proof}[Proof of Theorem~\ref{th:bara}]
Let $Z\subset \R^d$ denote the unit sphere of $\threebar{\cdot}$. Since $\threebar{\cdot}$ is a Barabanov norm, every point in $Z$ is the origin of at least one trajectory which takes values only in $Z$. Suppose for a contradiction that every point in $Z$ is the origin of a \emph{unique} such trajectory. Let $\mathfrak{X}_Z$ denote the set of all trajectories which take only values in $Z$, and observe that  by Proposition~\ref{pr:aff} $\mathfrak{X}_Z$ is a switched Bebutov shift. The function $\mathfrak{X}_Z \to S^{d-1}$ defined by $x \mapsto \|x(0)\|^{-1}x(0)$ is clearly continuous, and is bijective by hypothesis. It is therefore a homeomorphism. By Lemma~\ref{le:fixed-sphere} the shift semiflow $(\sigma^t)_{t \geq 0}$ on $\mathfrak{X}_Z$ has a fixed point $x$, and this trajectory satisfies $0=\dot{x}(t) = \sum_{i=1}^N u_i(t) A_i x(t)\text{ a.e.}$ for some measurable functions $u_1,\ldots,u_N \colon [0,\infty) \to [0,1]$ such that $\sum_{i=1}^Nu_i(t)=1$ a.e. In particular there exists a convex combination of the matrices $A_1,\ldots,A_N$ which has $x(0)$ in its kernel, which is a contradiction.\end{proof}

\section{Generalisations and directions for future work}

In this work we have considered only the case of arbitrary switching, in which the components $u_i$ of the switching law are Lebesgue measurable functions $[0,\infty) \to [0,1]$. However, Question~\ref{qu:siamreview} also makes sense when interpreted in terms of weak\nobreakdashes-* closed subclasses of those switching laws, such as the class of piecewise constant functions with a fixed minimum dwell time. In this context, each $u_i$ is instead a function $[0,\infty) \to \{0,1\}$; we again have $\sum_{i=1}^N u_i=1$ a.e.; and the vector $(u_1(t),\ldots,u_N(t))$ may take at most two distinct values in each interval of length $\delta$, where the minimum dwell time $\delta>0$ is some arbitrary constant. In this context segments of switching laws cannot be arbitrarily concatenated since this in general will violate the dwell time constraint, and thus trajectories do not necessarily satisfy the switching axiom of Definition \ref{de:switched}. While some of the preceding proofs can be adapted to this situation with minor modifications (such as Proposition~\ref{pr:minimal}) many other key steps break down, such as Proposition~\ref{pr:aff} and the final steps of Theorem~\ref{th:homog}. At this time we do not offer a conjecture on whether Question~\ref{qu:siamreview} has a positive answer for three-dimensional linear switched systems in the context of piecewise constant switching laws with a minimum dwell time, nor on whether a negative answer to the same question continues to hold in four dimensions and higher.

For simplicity of exposition we have restricted our attention in section~\ref{se:app} to the case in which the switched system is defined by a finite collection of states as opposed to a compact convex set of states. However, these results extend to this more general situation without deep modifications.

While we have shown in Theorem \ref{th:linear} that every marginally stable linear switched system in three dimensions has a periodic trajectory, we as yet know nothing of the structure of the underlying switching law. This contrasts with the two-dimensional case in which the structure of the periodic switching law is known very precisely: in this situation the switching law is either constant (in which case some convex combination of the switching states is non-invertible) or takes only extreme values $u_i(t) \in \{0,1\}$ and takes these values only on intervals. We therefore ask the following very general question:

\begin{question}
Let $\A=(A_1,\ldots,A_N) \in M_d(\R)^N$ be marginally stable with Barabanov norm $\threebar{\cdot}$. Suppose that there is no tuple $\B=(B_1,\ldots,B_M) \in M_d(\R)^M$ whose convex hull is strictly contained in that of $\A$ but which is also marginally stable. Does there exist a trajectory $x(t)$ along which $\threebar{x(t)}$ is constant and for which the underlying switching law $u=(u_1,\ldots,u_n)$ almost everywhere takes only values in $\{0,1\}^N$?
\end{question}

It follows from the results of e.g. \cite{PrMu25,PyRa96} that the answer to this question is positive in dimension $2$, and the answer is trivially positive in dimension $1$.

For questions specific to Bebutov shifts and the contents of Section~\ref{se:core}, see \S\ref{ss:Bebutov-questions}.

\bigskip

\noindent \textbf{Acknowledgements.} 
This research was partially supported by Leverhulme Trust Research Project Grant RPG-2024-380.
I.D.~Morris thanks the Pennsylvania State University for its hospitality during the visit in which this project was initiated.  The authors also thank Vaughn Climenhaga and Federico Rodriguez Hertz for helpful conversations.

\bibliographystyle{acm}
\bibliography{periodic} 
\end{document}